\documentclass[11pt, a4paper]{amsart}

\usepackage{amssymb,amsfonts}

\usepackage{mathtools}

\usepackage{amsmath,amsthm}

\usepackage{mathrsfs}

\usepackage{hyperref}

\usepackage[all]{xy}

\usepackage[active]{srcltx}

\usepackage{enumitem}

\allowdisplaybreaks

\numberwithin{equation}{section}

\newcommand{\Q}{\mathbb{Q}}
\newcommand{\N}{\mathbb{N}}
\newcommand{\R}{\mathbb{R}}
\newcommand{\Z}{\mathbb{Z}}

\newcommand{\CT}{\mathcal{T}}

\newcommand{\CA}{\mathfrak{A}}
\newcommand{\CB}{\mathcal{B}}
\newcommand{\CC}{\mathscr{C}}

\newcommand{\CP}{\mathscr{P}}
\newcommand{\CQ}{\mathscr{Q}}
\newcommand{\CS}{\mathscr{S}}

\newcommand{\CM}{\mathscr{M}}
\newcommand{\cB}{\mathcal{B}}

\newcommand{\fm}{\mathfrak{m}}

\newcommand{\s}{\sigma}

\newcommand{\vt}{\vartheta}
\newcommand{\ox}{\otimes}
\newcommand{\x}{\times}
\newcommand{\sm}{\setminus}
\newcommand{\smz}{\setminus \{0\}}
\newcommand{\Nil}{\mathrm{Nil}}
\newcommand{\wt}{\widetilde}
\newcommand{\id}{\mathrm{id}}
\newcommand{\vf}{\varphi}

\newcommand{\ul}{\underline}

\newcommand{\tP}{{\widetilde{P}}}
\newcommand{\tQ}{{\widetilde{Q}}}

\newcommand{\Bigmid}{\, \Big|\,}

\DeclareMathOperator{\im}{Im}

\DeclareMathOperator{\sign}{sign}
\DeclareMathOperator{\sgn}{sgn}
\DeclareMathOperator{\Sym}{Sym}

\DeclareMathOperator{\Int}{Int}

\DeclareMathOperator{\Trd}{Trd}

\DeclareMathOperator{\diag}{diag}

\DeclareMathOperator{\GL}{GL}
\DeclareMathOperator{\PSD}{PSD}
\DeclareMathOperator{\PD}{PD}

\renewcommand{\geq}{\geqslant}
\renewcommand{\leq}{\leqslant}

\renewcommand{\ge}{\geqslant}
\renewcommand{\le}{\leqslant}

\newcommand{\swap}{\widehat{\phantom{xx}}}

\newcommand{\df}{\emph}
\newcommand{\ve}{\varepsilon}

\newcommand{\op}{\mathrm{op}}

\newcommand{\pf}[1]{\langle\!\langle #1\rangle\!\rangle}

\newcommand{\qf}[1]{\langle #1\rangle}

\newcommand{\ns}{\mathrm{ns}}

\DeclareMathOperator{\Herm}{\mathfrak{Herm}}

\newtheorem{thm}{Theorem}[section]
\newtheorem{prop}[thm]{Proposition}
\newtheorem{cor}[thm]{Corollary}
\newtheorem{lemma}[thm]{Lemma}
\newtheorem*{claim}{Claim}
\theoremstyle{definition}
\newtheorem{defi}[thm]{Definition}

\newtheorem{remark}[thm]{Remark}
\newtheorem{ex}[thm]{Example}

\begin{document}

\title[Signature maps 
from positive cones]{Signature maps 
from positive cones on algebras with involution}

\author[V. Astier]{Vincent Astier}
\author[T. Unger]{Thomas Unger}
\address{School of Mathematics and Statistics, University College Dublin,
Belfield, Dublin 4, Ireland}
\email{vincent.astier@ucd.ie}
\email{thomas.unger@ucd.ie}


\maketitle

\begin{abstract}
  We introduced positive cones in an earlier paper as a notion of ordering
  on central simple algebras with involution that corresponds to
  signatures of hermitian forms. In the current paper we 
  describe signatures of hermitian forms directly out of positive
  cones, and also use this approach to rectify a problem that affected some results in
  the previously mentioned paper.
\end{abstract}

\tableofcontents

\section{Introduction}

In \cite{A-U-pos} we introduced the notion of positive cones for central
simple algebras with involution, inspired by the classical real algebra
of ordered fields. They are linked to signatures of hermitian forms, whose
investigation we started in \cite{A-U-kneb}, inspired by \cite{BP98}. 
We also gave a complete description of the kernels of the signatures maps
in \cite{A-U-prime}. 

In the current paper, after providing the necessary background in
Section~\ref{secprelim}, we show in Section~\ref{secsign} how to directly
obtain such a kernel out of a given positive cone. This construction also
allows us to rectify a recently discovered mistake in \cite{A-U-pos}.
Specifically, the mistake occurs in the proof of \cite[Lemma~5.5]{A-U-pos}, and
the lemma itself is likely incorrect. Hence, the lemmas in the remainder of
\cite[Section~5]{A-U-pos} (and their consequences) are potentially incorrect.
These lemmas are used to prove \cite[Proposition~5.8]{A-U-pos} which is the
only result in \cite[Section~5]{A-U-pos} that is used in the remainder of
\cite{A-U-pos}. In this paper we provide in particular an entirely different
proof of \cite[Proposition~5.8]{A-U-pos}, so that all the results in
\cite{A-U-pos} now have correct proofs, except for \cite[Lemmas~5.5, 5.6 and
5.7]{A-U-pos} which are no longer needed.

Note that in the process of reproving \cite[Proposition~5.8]{A-U-pos}, we
provide more direct proofs of results in \cite{A-U-pos} that were originally
obtained as consequences of \cite[Proposition~5.8]{A-U-pos}. Therefore, we
clearly indicate in each statement in Sections~\ref{secdesc} and \ref{secmax}
if it already appeared in \cite{A-U-pos}.

\section{Preliminaries}
\label{secprelim}

All fields in this paper are assumed to have characteristic different from $2$.
Let $F$ be such a field. We denote by $W(F)$ the Witt ring of $F$, by $X_F$ the
space of orderings of $F$, and by $F_P$ a real closure of $F$ at an ordering
$P\in X_F$.  We often denote the unique ordering on $F_P$ by $\tP$.

By an \emph{$F$-algebra with involution} we mean a pair
$(A,\s)$ where $A$ is a finite-dimensional simple $F$-algebra with centre a
field $K=Z(A)$, equipped with an involution $\s:A\to A$, such that $F = K \cap
\Sym(A,\s)$, where $\Sym(A,\s):=\{a\in A \mid \s(a)=a\}$. More
generally we let $\Sym_\ve(A,\s):=\{a\in A \mid \s(a)= \ve a\}$ for 
$\ve\in \{-1,+1\}$. If $A$ is a division algebra, we call $(A,\s)$ an 
\emph{$F$-division algebra with involution}.

Observe that $[K:F]\le 2$.
We say that 
$\s$ is \df{of the first kind} if $K=F$ and \df{of the second kind} (or 
\emph{of unitary type}) otherwise. 
Involutions of the first kind can be further subdivided into those of 
\emph{orthogonal type} and those of \emph{symplectic 
type}, depending on the dimension of $\Sym(A,\s)$, cf. 
\cite[Sections~2.A and 2.B]{BOI}.
We let $\iota=\s|_{K}$ and note that $\iota =\id_F$ if $\s$ is of the first 
kind. 

We denote by $W(A,\s)$ the Witt group  of Witt equivalence classes of 
nonsingular hermitian forms over $(A,\s)$ defined on 
finitely generated right $A$-modules. Note that $W(A,\s)$ is a $W(F)$-module.
We denote isometry of forms by $\simeq$. For  $a_1, \ldots, a_k$ in 
$\Sym(A,\s)$ the notation $\qf{a_1,\ldots, a_\ell}_\s$ stands for the diagonal 
hermitian form  
\[ \bigl(  (x_1,\ldots, x_\ell), (y_1,\ldots, y_\ell)  \bigr) \in A^\ell \x 
A^\ell \mapsto  
 \sum_{i=1}^\ell \s(x_i) a_i y_i \in A.\]
We often identify nonsingular quadratic and hermitian forms with their
Witt classes if no confusion is possible. If $h$ is a hermitian form over
$(A,\s)$, we denote the set of elements represented by $h$ by $D_{(A,\s)}(h)$.

We denote by $\Int(u)$ the inner automorphism determined by $u \in A^\x$, i.e., 
$\Int(u)(x):= uxu^{-1}$ for $x\in A$. We also write $\s^t$ for the involution
$(a_{ij})\mapsto (\s(a_{ji}))$ on $M_n(A)$.

\subsection{Signatures of hermitian forms over quadratic field extensions and
quaternions}
\label{secqq}

Let $k$ be a field  
and let $D_k \in \{k, k(\sqrt{d}), (a,b)_k\}$,
where $a,b,d\in k$,
$k(\sqrt{d})\not=k$ and $(a,b)_k$ is a quaternion division algebra over 
$k$. Let $\vt_k$ denote the canonical involution on $D_k$,
i.e., the identity map on $k$ or conjugation in the remaining cases. 

If $h$ is a
hermitian form over $(D_k,\vt_k)$, then $b_h(x):=h(x,x)$ is a quadratic form
over $k$. A straightforward computation shows that if 
$h\simeq\qf{a_1,\ldots,a_n}_{\vt_k}$ with $a_1,\ldots,a_n \in 
\Sym(D_k,\vt_k)= k$, then
\begin{equation}\label{sign0}
  b_h \simeq \begin{cases} \qf{a_1,\ldots,a_n}  & \text{if } 
  D_k=k \\
  \qf{1,-d}\ox \qf{a_1,\ldots,a_n}  & \text{if } 
  D_k=k(\sqrt{d}) \\
  \qf{1,-a,-b,ab}\ox \qf{a_1,\ldots,a_n} & \text{if } 
  D_k=(a,b)_k \end{cases}.
\end{equation}
By Jacobson's theorem (cf. \cite[Chapter~10, Theorems~1.1 and 1.7,
Remark~1.3]{Sch}) the map 
\[
  W(D_k,\vt_k) \rightarrow W(k),\ h \mapsto b_h
\] 
is injective. 

Let $P\in X_k$. The preceding paragraph motivates defining the signature of
$h$ at $P$ in terms of the Sylvester signature $\sign_P b_h$ as follows:
\begin{equation}\label{sign1}
  \sign_P h := \begin{cases} \sign_P b_h  & \text{if } 
  D_k=k \\
    \tfrac{1}{2} \sign_P b_h & \text{if }  D_k=k(\sqrt{d})\text{ with }d<_P 0 \\
   \tfrac{1}{4} \sign_P b_h & \text{if } D_k=(a,b)_k \text{ with }a,b <_P 0\\
   0 & \text{in all remaining cases}
  \end{cases}.
\end{equation}

\begin{remark}
  If $D_k=k(\sqrt{d})$, skew-hermitian forms over $(D_k,\vt_k)$ are equivalent
  to hermitian forms, cf. \cite[Lemma~2.1(iii)]{A-U-kneb}.

  If $D_k \in \{k,(a,b)_k\}$ and $h$ is a skew-hermitian form over
  $(D_k,\vt_k)$, or a hermitian or skew-hermitian form over $(D_k\x D^\op_k,
  \swap)$ with $\widehat{(x,y^\op)}:=(y,x^\op)$ the exchange involution, then
  (the nonsingular part of) $h$ is torsion in the Witt group and we let
  \begin{equation}\label{sign2} \sign_P h:=0, \end{equation} cf.
  \cite[Section~3.1 and Lemma~2.1]{A-U-kneb}.
\end{remark}

\subsection{Signatures of hermitian forms over $F$-algebras with involution}
\label{csa-sign}

Returning to the general case of an $F$-algebra with involution $(A,\s)$, let
$P\in X_F$ and let   
\[  
 (D_P,\vt_P):=(D_{F_P}, \vt_{F_P}) \in \{ (F_P, \id), (F_P(\sqrt{-1}),-),
((-1,-1)_{F_P},-)\},
\]  
using the notation
from Section~\ref{secqq}.  
We define the signature
of a hermitian form $h$ over $(A,\s)$ by extending scalars to $F_P$. Write
$Z(A)=F(\sqrt{d})$ with $d\in F$. We consider two cases:

(1) If $\s$ is of the second kind and $d>_P0$,
then $Z(A)\ox_F F_P \cong F_P\x F_P$ and
we obtain
\begin{equation}\label{sign4}
  (A\ox_F F_P, \s\ox\id) \cong (M_{n_P}(D_P)\x M_{n_P}(D_P)^\op, \swap),
\end{equation}
cf. \cite[Proposition~2.14]{BOI}. Since (the nonsingular part of) $h$ is
zero in the Witt group (cf.  \cite[Lemma~2.1(iv)]{A-U-kneb}),  
we will define the signature of $h$ at $P$
  to be zero in this case, cf. \eqref{sign6} below.

(2) If $\s$ is of the second kind and $d<_P0$, or if $\s$ is of the first kind,
 then by the Skolem-Noether theorem we obtain an isomorphism of
$F_P$-algebras with involution
\begin{equation}\label{sign3}
  (A\ox_F F_P, \s\ox\id) \cong (M_{n_P}(D_P), \Int(\Phi_P)\circ {\vt_P}^t ),
\end{equation}
where $\Phi_P \in \Sym_{\ve} (M_{n_P}(D_P), {\vt_P}^t)$ is invertible and
$\ve=1$ if $\s$ and $\vt_P$ are of the same type and $\ve=-1$ otherwise, cf.
\cite[Propositions 2.7 and 2.18]{BOI}.

The $F_P$-algebra with involution $(M_{n_P}(D_P), \Int(\Phi_P)\circ {\vt_P}^t )$
is hermitian Morita equivalent to $(M_{n_P}(D_P), {\vt_P}^t )$ (via scaling by
$\Phi_P^{-1}$), which in turn is hermitian Morita equivalent to $(D_P, \vt_P)$,
cf. \cite[Section~2.4]{A-U-kneb}.
We denote the composition of these equivalences, and its induced
map on hermitian forms, by $\fm_P$.

\begin{remark}\label{BS}
  Observe that if $\s$ is orthogonal and $D_P=(-1,-1)_{F_P}$ or if
  $\s$ is symplectic and $D_P=F_P$, then $\ve = -1$ and $\fm_P(h \ox_F
  F_P)$ is skew-hermitian over $(D_P, \vt_P)$. Therefore, in accordance with 
  Remark~2.1, we will define the signature of $h$ at $P$
  to be zero in this case, cf. \eqref{sign6} below.
\end{remark}

\begin{defi}[See also \eqref{sign7} below]
  We say that $P$ is a \emph{nil-ordering of $(A,\s)$} if \eqref{sign4} 
  holds or if one of the cases described in
  Remark~\ref{BS} occurs. We denote the set of nil-orderings of $(A,\s)$ by
  $\Nil[A,\s]$, where the square brackets indicate that this set depends only on
  the Brauer class of $A$ and the type of $\s$.
\end{defi}

Assume now that $P\in X_F\sm \Nil[A,\s]$.  As already mentioned, the idea is to
define the signature of $h$ at $P$
as $\sign_\tP \fm_P(h \ox_F F_P)$ via \eqref{sign1}, where
$\tP$ denotes the unique ordering on $F_P$. There is however a problem: while a
different choice of real closure does not affect this definition (cf.
\cite[Proposition~3.3]{A-U-kneb}) there is no canonical choice of Morita
equivalence, and different choices can result in sign changes (cf.
\cite[Proposition~3.4]{A-U-kneb}). This problem can be addressed as follows: we
showed
in \cite[Theorem~6.4]{A-U-kneb} and
\cite[Sections~2 and 3]{A-U-prime} that there
exists a hermitian form $\mu$ over $(A,\s)$, called a reference form
for $(A,\s)$, with the property that the
signature of the hermitian form $\fm_Q (\mu\ox F_Q)$ 
over $(D_Q,\vt_Q)$
is nonzero at all  $Q\in X_F\sm \Nil[A,\s]$.
Let $s_P\in\{-1,1\}$ denote the sign of 
$\sign_\tP \fm_P (\mu\ox F_P)$. 
The $\mu$-signature of $h$ at $P$ is then defined as
\begin{equation}\label{sign5}
  \sign^\mu_P h:= s_P\cdot \sign_\tP \fm_P(h\ox F_P).
\end{equation}
This definition  ensures
that the use of different Morita equivalences does not change the 
result, cf. \cite[Lemma~3.8]{A-U-kneb}.
The choice of a different reference
form may result in $\sign^\mu_P h$ changing sign continuously
at all $P\in X_F\sm \Nil[A,\s]$, cf. \cite[Proposition~3.3(iii)]{A-U-prime}.
\medskip

Finally, if $P\in \Nil[A,\s]$ 
we define
\begin{equation}\label{sign6}
  \sign^\mu_P h:= 0.
\end{equation}
Note that by \cite[Theorem~6.4]{A-U-kneb} we actually have
\begin{equation}\label{sign7}
  P\in \Nil[A,\s] \Leftrightarrow \sign^\mu_P=0.
\end{equation}

\begin{remark}\label{blobby}
  Observe that if $P\in X_F\sm \Nil[A,\s]$, then there exists 
  $a\in\Sym(A,\s)\cap A^\x$ such that $\sign^\mu_P \qf{a}_\s\not=0$.
  Indeed, if $h$  is such that $\sign^\mu_P h \not=0$, this follows from
  ``weak diagonalization'' (cf.  \cite[Lemma~2.2]{A-U-pos}) and the fact
  that $\sign^\mu_P$ is additive.  
\end{remark}

\begin{remark}
  The definition of $\mu$-signature implies that
  \[
    \sign_P^\mu h =\sign_\tP^{\mu\ox F_P} (h\ox F_P).
  \]
  Furthermore, if $(A,\s)=(D_F,\vt_F)$ (with notation as in 
  Section~\ref{secqq}), then $\mu:=\qf{1}_\s$ is a reference form for $(A,\s)$
  (since $\sign_P \qf{1}_\s=1$ for all $P\in X_F\sm \Nil[A,\s]$) and
  \[
    \sign_P^\mu = \sign_P.
  \]
  
\end{remark}

\noindent\textbf{Assumption for the remainder of the paper:}
 $(A,\s)$ is an $F$-algebra with 
involution and $\mu$ is a reference form for $(A,\s)$.

\subsection{Signatures under ordered field embeddings}

We recall the following consequence of \cite[Lemma~4.1]{A-U-prime}:
\begin{lemma}\label{trivial1}
  Let $F_P \subseteq L$ be a field extension with $L$ real closed. We denote by
  $\fm$ the hermitian Morita equivalence between $(A \ox_F F_P,\s \ox \id)$ and 
  $(D_P, \vt_P)$ as well as the induced isomorphism of Witt groups. Then $\fm$
  extends to a hermitian Morita equivalence $\fm'$ between $(A\ox_F F_P
  \ox_{F_P} L, \s \ox \id \ox \id) = (A\ox_F L, \s \ox \id)$ and $(D_P \ox_{F_P}
  L, \vt_P \ox \id)$ such that (denoting  the induced isomorphism of
  Witt groups also by $\fm'$), the following diagram commutes:
  \[\xymatrix{
    W(A \ox_F F_P, \s \ox \id) \ar[r]^\fm \ar[d] & W(D_P, \vt_P) \ar[d] \\
    W(A \ox_F F_P \ox_{F_P} L, \s \ox \id \ox \id) \ar[r]^-{\fm'} & W(D_P
    \ox_{F_P} L, \vt_P \ox \id)
  }\]
\end{lemma}

\begin{lemma}\label{trivial2}
  Let $(D,\vt) \in \{(F, \id), (F(\sqrt{-1}), -), ((-1,-1)_F,-)\}$ with $F$ real
  closed. Let $(F,P) \subseteq (L,Q)$ be an extension of ordered fields
  with $L$ real 
  closed, and let
  $h$ be a nonsingular  hermitian form over $(D,\vt)$. Then 
  \[
   \sign_P h = \sign_Q (h \ox L).
  \] 
\end{lemma}

\begin{proof}
  Since $D$ is a division algebra, $h$ can be diagonalized with entries from
  $\Sym(D,\vt)=F$. Since $h$ is nonsingular and
  $F$ is real closed, we  have
  $h \simeq r \x \qf{1}_\vt \perp s \x \qf{-1}_\vt$, and so $\sign_P h =
  r-s$. Then $h \ox L \simeq r \x \qf{1}_{\vt\ox\id} \perp s \x 
  \qf{-1}_{\vt\ox\id}$, so that
  $\sign_Q (h \ox L) = r-s$.
\end{proof}

\begin{lemma}\label{trivial3}
  Let $P \in X_F$ and let  $\lambda : (F_P, \tP)\to (L,Q)$ be an embedding
  of ordered fields with $(L, Q)$ real closed.
 Let $h$ be a
  nonsingular hermitian form over $(A,\s)$. Then
  \[
   \sign_P^\mu h = \sign_Q^{\mu\ox_\lambda L} (h \ox_\lambda L).
  \]
\end{lemma}

\begin{proof}
  The proof has two parts.
  
  \emph{Part 1:} Assume that $\lambda$ is an inclusion.  
  Observe that by Lemma~\ref{trivial1}, $\fm(h \ox F_P) \ox_{F_P} L = \fm'(h
  \ox_F F_P \ox_{F_P} L) = \fm'(h \ox L)$ in $W(D_P \ox_{F_P} L, \vt_P
  \ox \id) \cong W(D_L, \vt_L)$. 
  Let $s_P$ and $s_Q\in\{-1,+1\}$ denote the sign of $\sign_{\tP} \fm(\mu \ox 
  F_P)$
  and $\sign_{Q} \fm'((\mu \ox F_P) \ox L)$, respectively.
  Observe that $s_P=s_Q$ by Lemma~\ref{trivial1}.  It follows that
  \begin{align*}
    \sign_P^\mu h &= s_P\cdot \sign_{\tP} \fm(h \ox F_P) \text{ by 
        \eqref{sign5}} \\
      &= s_P \cdot \sign_Q  (\fm(h \ox F_P) \ox L) \text{ by 
        Lemma~\ref{trivial2}} \\
      &= s_Q \cdot \sign_Q \fm'( (h \ox F_P)\ox L) \text{ by 
        Lemma~\ref{trivial1}} \\ 
     &= \sign_Q^{\mu\ox L} (h \ox L) \text{ by 
        \eqref{sign5}}.
  \end{align*}
    
  \emph{Part 2:} Returning to the general case of a morphism 
  $\lambda : (F_P, \tP)\to (L,Q)$,
  we have
  \[
   \sign_P^\mu h = \sign_Q^{\mu\ox_\lambda L} (h \ox_\lambda L)
  \]
  since both 
  $h\ox_{\lambda} L$ and $\mu\ox_\lambda L$ are obtained by applying
  the isomorphism
  $\lambda: F_P \to \lambda(F_P)$, which preserves signatures
  by \cite[Theorem~4.2]{A-U-prime}, followed by
  the inclusion $\lambda(F_P)\subseteq L$, which also preserves signatures
  by the argument above.
\end{proof}

\begin{thm}\label{embord}
  Let $h$ be a hermitian form over $(A,\s)$ and let $P \in X_F$. 
  Let $\lambda: (F,P) \to (L,Q)$ be an embedding of ordered fields.
  Then
  \[
    \sign_P^\mu h = \sign_Q^{\mu\ox_\lambda L} (h \ox_\lambda L).
  \]
\end{thm}

\begin{proof}
  We may assume that $h$ is nonsingular 
  since otherwise we can write $h\simeq h^\ns\perp h_0$, 
  where $h^\ns$ is nonsingular and $h_0$ is a zero form of appropriate rank,
  cf. \cite[Proposition~A.3]{A-U-PS}, and thus $\sign_P^\mu h = 
  \sign_P^\mu h^\ns$.
    
  \emph{Part 1:} Assume that $\lambda$ is an inclusion.  
  Let $(L_Q, \tQ)$ be a real closure of $(L,Q)$.
  By \cite[Exercise~1.4.3(b)]{S24} 
  there is a real closed field $(N, S)$ and embeddings of ordered fields
  $\lambda_P$ and $\lambda_Q$ such that the following diagram commutes:
  \begin{equation}\label{amal}
    \begin{gathered}
      \xymatrix{
            & (F_P, \tP)\ar[dr]^{\lambda_P} & \\
       (F,P)\ar[ur]\ar[dr]  &            & (N,S)\\
            & (L_Q, \tQ)\ar[ur]_{\lambda_Q} &
      }
    \end{gathered}
  \end{equation}
  (This can also be obtained as a consequence of elimination of quantifiers
  for real closed fields by \cite[Proposition~3.5.19]{CK}.)  
  By definition, 
  \[
    \sign_P^\mu h = \sign_{\tP}^{\mu\ox F_P} (h \ox F_P)
  \] 
  and 
  \[
  \sign_Q^{\mu\ox L} (h \ox L) = \sign_{\tQ}^{(\mu \ox L) \ox L_Q} 
  ((h \ox L) \ox L_Q).
  \] 
  By Lemma~\ref{trivial3} we have 
  \[
  \sign_{\tP}^{\mu\ox F_P} (h \ox F_P) = 
  \sign_S^{(\mu\ox F_P) \ox_{\lambda_P} N} 
  ((h \ox F_P)\ox_{\lambda_P}   N)
  \] 
  and 
  \[
    \sign_{\tQ}^{\mu \ox L \ox L_Q} 
      (h \ox L \ox L_Q) = \sign_S^{(\mu\ox L\ox L_Q) \ox_{\lambda_Q} N}  
      ((h \ox L\ox L_Q ) 
      \ox_{\lambda_Q} N).
  \]  
 The result follows,
  since 
  $(h \ox F_P)\ox_{\lambda_P}   N \cong  (h \ox L\ox L_Q ) \ox_{\lambda_Q} N$
  and
  $(\mu \ox F_P)\ox_{\lambda_P}   N \cong  (\mu \ox L\ox L_Q)  
  \ox_{\lambda_Q} N$
  by commutativity of diagram \eqref{amal}.
  
  \emph{Part 2:} Assume that $\lambda$ is any embedding. We conclude
  as in Part~2 of the proof of Lemma~\ref{trivial3}.  
\end{proof}

\subsection{Positive cones}

Positive cones on algebras with involution were introduced in \cite{A-U-pos}
as an attempt to define a notion of ordering that corresponds to signatures of
hermitian forms and that has good real-algebraic properties.

\begin{defi}[{\cite[Definition~3.1]{A-U-pos}}]\label{def-preordering}
 A \emph{prepositive cone $\CP$ on }$(A,\s)$ is a subset $\CP$ of $\Sym(A,\s)$
  such that
  \begin{enumerate}[label=(P\arabic*)]
    \item $\CP \not = \varnothing$;
    \item $\CP + \CP \subseteq \CP$;
    \item $\s(a) \cdot \CP \cdot a \subseteq \CP$ for every $a \in A$;
    \item $\CP_F := \{u \in F \mid u\CP \subseteq \CP\}$ is an ordering on $F$;
    \item $\CP \cap -\CP = \{0\}$ (we say that $\CP$ is \emph{proper}).
  \end{enumerate}
  A prepositive cone $\CP$ is \emph{over} $P\in X_F$ 
  if $\CP_F=P$, and a \emph{positive cone} is a prepositive cone that is 
  maximal with respect to inclusion.
  We denote the set of all positive cones on $(A,\s)$ by $X_{(A,\s)}$.
\end{defi}

Note that $\CP$ is a (pre)positive cone over $P$ if and only if $-\CP$ is a (pre)positive cone over $P$.

\begin{ex}
  The simplest non-trivial example of a positive cone is given by the set of
  positive semidefinite matrices in any of the following central simple
  algebras with involution: \[(M_n(\R), t),\ (M_n(\R(\sqrt{-1})), -^t),\
  (M_n((-1,-1)_{\R}), -^t)\] (see \cite[Example~3.11 and Remark~4.11]{A-U-pos}
  for the case of $(M_n(\R), t)$; the exact same argument works for the other
  two cases, using the principal axis theorem, which also holds for matrices
  over quaternions by \cite[Corollary~6.2]{Zhang}).
\end{ex}

\begin{defi}
  Let $S \subseteq \Sym(A,\s)$ and let $P \in X_F$. We define
  \[\CC_P(S) := \Bigl\{\sum_{i=1}^k u_i \s(x_i)s_ix_i \Bigmid 
  k \in \N,\ u_i \in P,
  \ x_i \in A,\ s_i \in S \Bigr\},\]
  and for $a\in \Sym(A,\s)$ and $\CP$ a prepositive cone on $(A,\s)$ over $P$,
  \[\CP[a] := \Bigl\{p + \sum_{i=1}^k u_i\s(x_i)ax_i \Bigmid p \in \CP,\ k
  \in \N,\ u_i \in P,\ x_i \in A\Bigr\}.\]
\end{defi}

It is clear that $\CC_P(S)$ and $\CP[a]$ both
satisfy properties (P1) to (P4),
and are prepositive cones over $P$
if and only if they are proper, i.e., satisfy (P5).
\medskip

\begin{defi}
  We define, for $P \in X_F$,
  \[m_P(A,\s) := \max \{\sign^\mu_P \qf{a}_\s \mid a \in \Sym(A,\s)\cap A^\x\}\]
  and, for $P \in X_F \sm \Nil[A,\s]$,
  \[\CM^\mu_P(A,\s) := \{a \in \Sym(A,\s)\cap A^\x \mid \sign^\mu_P \qf{a}_\s =
  m_P(A,\s)\} \cup \{0\}.\]
\end{defi}

Observe that if $P \in X_F \sm \Nil[A,\s]$ then $m_P(A,\s)>0$ and so
$\CM^\mu_P(A,\s)\not=\{0\}$, by Remark~\ref{blobby}.

\begin{prop}\label{lausanne} Let $P \in X_F\sm \Nil[A,\s]$. 
  If $A$ is an $F$-division algebra, then $\CM^\mu_P(A,\s)$ is a
  prepositive cone on $(A,\s)$ over $P$. Otherwise, 
  $\CC_P(\CM^\mu_P(A,\s))$ is a prepositive cone over $P$.  
\end{prop}

\begin{proof}
  The first statement is \cite[Example~3.13]{A-U-pos}. For the second statement,
  it suffices to check that $\CC_P(\CM^\mu_P(A,\s))$ is proper, since properties
  (P1) to (P4) are clear.  Assume that this is not the case. Then
  $\CC_P(\CM^\mu_P(A,\s)) = \Sym(A,\s)$ by \cite[Proposition 3.5]{A-U-pos}. In
  particular there are elements $a_1,\ldots, a_r$, $b_1,\ldots, b_s\in
  \CM^\mu_P(A,\s)\smz$ such that $1 \in D_{(A,\s)} \qf{a_1,\ldots, a_r}_\s$ and
  $-1 \in D_{(A,\s)} \qf{b_1,\ldots, b_s}_\s$. Since both $1$ and $-1$ are
  invertible, a standard argument shows that $\qf{1}_\s \perp \vf \simeq
  \qf{a_1,\ldots, a_r}_\s$ and $\qf{-1}_\s \perp \psi \simeq \qf{b_1,\ldots,
  b_s}_\s$ for some nonsingular hermitian forms $\vf$ and $\psi$ over $(A,\s)$.
  Therefore, $\qf{1,-1}_\s \perp \vf \perp \psi \simeq \qf{a_1,\ldots, a_r,
  b_1,\ldots, b_s}_\s$. By ``weak diagonalization'', 
  cf.  \cite[Lemma~2.2]{A-U-pos},
  we have
  \[\ell \x \qf{1,-1}_\s \perp \qf{c_1, \ldots, c_k}_\s \simeq \ell \x 
  \qf{a_1,\ldots, a_r, b_1,\ldots, b_s}_\s\]
  for some $\ell\in \N$,
  $c_1,\ldots, c_k \in \Sym(A,\s)\cap A^\x$, and $2\ell +k = \ell(r+s)$.
  Comparing signatures at $P$, we obtain that the right-hand side has signature
  $\ell (r+s) \cdot m_P(A,\s)$ (with $m_P(A,\s) > 0$ since 
  $P \not \in \Nil[A,\s]$), which is the
  maximal value that can be obtained by the signature of a diagonal form of
  dimension $\ell(r+s)$. But the left-hand side can only have signature at most
  $k \cdot m_P(A,\s)$, which is smaller than $\ell (r+s) \cdot m_P(A,\s)$, 
  contradiction.
\end{proof}

\subsection{Reduction to diagonal forms}

We recall from \cite[Section~3.4]{A-U-Az-PLG} 
that there exists a pairing $*$ of hermitian forms over $(A,\s)$ 
(first studied in detail by N.~Garrel in 
\cite{garrel-2023})
such
that $\vf * \psi$ is a hermitian form over $(Z(A),\iota)$, where 
$\iota:=\s|_{Z(A)}$, and which preserves orthogonal sums, isometries and
nonsingularity, cf. \cite[Corollary~3.8]{A-U-Az-PLG}. Furthermore, $*$ 
satisfies the following  ``pivot property''
\begin{equation}\label{pivot}
  (\vf * \psi) \ox_{Z(A)} \chi \simeq  (\chi * \psi ) \ox_{Z(A)} \vf,
\end{equation}
cf. \cite[Theorem~3.9]{A-U-Az-PLG}. We also note that if 
$a,b \in \Sym(A,\s)\cap A^\x$, then by \cite[Proposition~4.9]{garrel-2023} or 
\cite[Lemma~3.11]{A-U-Az-PLG} we have
\[
  \qf{a}_\s * \qf{b}_\s \simeq \vf_{a,b,\s},
\]
where $\vf_{a,b,\s}(x,y):=\Trd_A(\s(x)ayb)$.
\medskip

Observe that by \cite[Lemma~3.6]{A-U-pos} there exists an invertible element
$a$ in $\CP$.

\begin{prop}\label{anonzero}
  Let $\CP$ be a positive cone on $(A,\s)$ over $P \in X_F$ and let 
  $a\in \CP\cap A^\x$. Then $\sign_P^\mu \qf{a}_\s\not=0$ and
  $\sign_P (\qf{a}_\s * \qf{a}_\s)\not=0$.
\end{prop}

\begin{proof}
  Assume for the sake of contradiction that $\sign_P^\mu \qf{a}_\s=0$.
  By continuity of the total signature map $\sign^\mu_\bullet \qf{a}_\s$ (cf. 
  \cite[Theorem~7.2]{A-U-kneb}), there exist $u_1,\ldots, u_k \in F^\x$ such
  that $P$ belongs to the Harrison set $H(u_1,\ldots, u_k)$  and 
  $\sign^\mu_\bullet \qf{a}_\s=0$
  on $H(u_1,\ldots, u_k)$. Consider the Pfister form 
  $\pf{u_1,\ldots, u_k}:= \qf{1,u_1}\ox \cdots \ox \qf{1,u_k}$. Then we have
  $\sign_Q^\mu  \pf{u_1,\ldots, u_k} \ox \qf{a}_\s=0 $, for all $Q\in X_F$.
  It then follows from Pfister's local-global principle 
  (cf. \cite[Theorem~4.1]{L-U} or \cite[Theorem~6.5]{B-U})
  that there
  exists $n\in \N$ such that $2^n \x \pf{u_1,\ldots, u_k} \ox \qf{a}_\s$ is
  hyperbolic. Since this form is a diagonal form with $2^{k+n}$
  entries we can write it as a sum of hyperbolic planes as follows:
  \[
    2^n\x\pf{u_1,\ldots, u_k} \ox \qf{a}_\s\simeq 2^{k+n-1}\x \qf{-a,a}_\s.
  \]
  In particular, $-a$ is represented by the form on the left-hand side and so
  $a\in -\CP$, which contradicts that $\CP$ is proper.  
  
  Next we prove the second statement. 
  Since $\s(a^{-1})aa^{-1}=a^{-1}aa^{-1}=a$, we have $\qf{a}_\s\simeq 
  \qf{a^{-1}}_\s$. Therefore,
  \[
    \qf{a}_\s * \qf{a}_\s \simeq \qf{a}_\s * \qf{a^{-1}}_\s\simeq 
    \vf_{a,a^{-1},\s}= T_{(A,\s_a)},
  \]
  where $\s_a:=\Int(a^{-1})\circ \s$ and 
  \[
  T_{(A,\s_a)}(x,y):=\Trd_A(\s_a(x)y)=\Trd_A(a^{-1}\s(x)ay).
  \]
  It then follows from \cite[Equation~(4.1) and Proposition~4.4(i)]{A-U-PS} that
  \[
    \sign_P(\qf{a}_\s * \qf{a}_\s) =\sign_P T_{(A,\s_a)} = (\sign_P \s_a)^2
    ={\lambda_P}^2 (\sign_P^\mu \qf{a}_\s)^2,
  \]
  where $\lambda_P\not=0$.   
\end{proof}

\begin{prop}\label{Sylv}
  Let $\CP$ be a positive cone on $(A,\s)$ over $P \in X_F$
  and let $a\in \CP\cap A^\x$. For every
  nonsingular hermitian form $\vf$ over $(A,\s)$ there exists a nonsingular 
  quadratic form $q$ over $F$ such that $\sign_P q\not=0$ and
  \[
    q\ox \vf \simeq (\qf{u_1,\ldots, u_r} \perp \qf{-v_1,\ldots, -v_s})\ox 
    \qf{a}_\s
  \]
  for some $u_1,\ldots, u_r,v_1,\ldots, v_s \in P\smz$.  
\end{prop}

\begin{proof}
  Using 
  \eqref{pivot}, we have
  \begin{equation}\label{qfprod}
    (\qf{a}_\s * \qf{a}_\s) \ox_{Z(A)} \vf \simeq (\vf * \qf{a}_\s) \ox_{Z(A)}
    \qf{a}_\s. 
  \end{equation}
  The forms $\qf{a}_\s * \qf{a}_\s$ and $\vf * \qf{a}_\s$ are both 
  nonsingular
  hermitian
  over $(Z(A), \iota)$,  and are thus diagonalizable
  with coefficients in $\Sym(Z(A), \iota)\cap Z(A)^\x=F^\x$. 
  Hence there exist $w_1,\ldots, w_t\in F^\x$ such that 
  $\qf{a}_\s * \qf{a}_\s\simeq \qf{w_1,\ldots, w_t}_\iota$, and there exist
  $u_1,\ldots, u_r, v_1,\ldots, v_s \in P\smz$ such that
  \[
    \vf * \qf{a}_\s \simeq \qf{u_1,\ldots, u_r}_\iota \perp 
    \qf{-v_1,\ldots, -v_s}_\iota.
  \]
  By \cite[Lemma~1.1]{A-U-Az-PLG} applied to \eqref{qfprod} we obtain
  \begin{align*}
    \qf{w_1,\ldots, w_t} \ox_F \vf
    &\simeq \qf{w_1,\ldots, w_t}_\iota \ox_{Z(A)} \vf 
    \\ 
    &\simeq (\qf{u_1,\ldots, u_r}_\iota \perp 
    \qf{-v_1,\ldots, -v_s}_\iota) \ox_{Z(A)} \qf{a}_\s\\ 
    &\simeq
    (\qf{u_1,\ldots, u_r} \perp 
        \qf{-v_1,\ldots, -v_s}) \ox_F \qf{a}_\s.     
  \end{align*} 
  It follows from Proposition~\ref{anonzero} that 
  $\sign_P \qf{w_1,\ldots, w_t}_\iota \not=0$, and thus that
  $\sign_P \qf{w_1,\ldots, w_t} \not=0$ by \eqref{sign0} and \eqref{sign1}.
\end{proof}

\begin{lemma}\label{samenr}
  Let $\CP$ be a positive cone on $(A,\s)$ over $P \in X_F$.
  In an isometry of diagonal hermitian forms with coefficients in
  $\CP \cap A^\x$ and $-\CP \cap A^\x$, if there are as many elements in
  $\CP$ as in $-\CP$ on one side, it must be the same on the other side.
\end{lemma}

\begin{proof}
  Assume that, for some $a_i, b_i,
  c_i, d_i \in \CP\cap A^\x$ we have
  \[\qf{a_1, \ldots, a_r}_\s \perp \qf{-b_1, \ldots, -b_r}_\s \simeq
  \qf{c_1, \ldots, c_s}_\s \perp \qf{-d_1, \ldots, -d_t}_\s,\]
  with, for instance, $s > t$. Then $s > r > t$ and
  \begin{align*}
    \qf{c_1, \ldots, c_s}_\s \perp \qf{b_1, \ldots, b_t}_\s \simeq
    &\qf{a_1, \ldots, a_r}_\s \perp \qf{d_1, \ldots, d_t}_\s \perp \\
    &\quad \qf{-b_{t+1}, \ldots, -b_r}_\s.
  \end{align*}
  Since the entries on the left-hand side are all in $\CP\cap A^\x$, 
  the entries
  on the right-hand side must be in $\CP$ (they are represented by the first
  form, and $\CP$ is closed under the operations presented in properties (P2)
  and (P3)). In particular $-b_r \in \CP$, which
  contradicts that $\CP$ is proper.
\end{proof}

\section{Signature maps from positive cones}
\label{secsign}

Consider a positive cone $\CP$ on $(A,\s)$
over $P\in X_F$.
In this section we will define the signature map $\sign_P^\mu W(A,\s)
\rightarrow \Z$ directly out of $\CP$ via the concept of 
prime m-ideals, that was introduced in \cite{A-U-prime}, and that we recall
now.

\begin{defi}[{\cite[Definition~5.1]{A-U-prime}}]
  We say that a pair $(I,N)$ is an \emph{m-ideal} of $W(A,\s)$ if:
  \begin{enumerate}
    \item $I$ is an ideal of $W(F)$ and $N$ is a $W(F)$-submodule 
    of $W(A,\s)$;
    \item $I \cdot W(A,\s) \subseteq N$.
  \end{enumerate}
  In addition we say that the m-ideal $(I,N)$ is \emph{prime} if 
  $I$ is a proper prime ideal of $W(F)$,
  $N\not=W(A,\s)$ and, 
  for every $q\in W(F)$ and every $h\in W(A,\s)$, $q\cdot h \in N$ implies
  that $q \in I$ or $h \in N$.
\end{defi}

We recall \cite[Proposition~6.5]{A-U-prime}:

\begin{prop}
  Let $(I,N)$ be a prime m-ideal of $W(A,\s)$ such that $2 \not \in I$ and
  $W(A,\s)/N$ is torsion-free. Then there exists $P \in X_F$ such that $(I,N) =
  (\ker \sign_P,  \ker \sign_P^\mu)$.
\end{prop}

We will define a prime m-ideal $(I_\CP, N_\CP)$ such that $2\not\in I_\CP$ and
the quotient
$W(A,\s)/N_\CP$ is torsion-free directly from $\CP$, thus recovering the
signature map $\sign^\mu_P$  out of the positive cone $\CP$.

Therefore, and since we will ultimately have $N_\CP = \ker \sign^\mu_P$,
we need to determine the nonsingular hermitian forms over $(A,\s)$
that are good candidates for having zero signature at $P$, and use their Witt
classes as elements of $N_\CP$.

\begin{defi}\label{N}
  For a hermitian
  form $h$ over $(A,\s)$ we define the following property:
  \begin{equation}\label{Z}
  \left\{\begin{array}{l}
    \text{There}  \text{ exists} \text{ a nonsingular 
    quadratic form $q_h$ over $F$}\\
    \text{such that}\\
    \quad\bullet\ \sign_P q_h \not = 0 \text{ and}\\
    \quad\bullet\ q_h \ox h \simeq \qf{a_1, \ldots, a_r}_\s \perp \qf{-b_1, \ldots,
    -b_r}_\s\text{ for some }\\
    \qquad r\in \N \text{ and }
    a_1, \ldots, a_r, b_1, \ldots, b_r \in \CP \cap A^\x.
  \end{array}\right.
  \end{equation}
\end{defi}

\begin{lemma}\label{witteq} 
  Property~\eqref{Z} is preserved under Witt equivalence
  (of nonsingular forms).  
\end{lemma}

\begin{proof}
  Let $h$ be a nonsingular hermitian form over $(A,\s)$ that satisfies 
  Property~\eqref{Z}. Let $h'$ be a nonsingular hermitian form over $(A,\s)$ 
  such that 
  \begin{equation}\label{camargue}
    h\perp H  \simeq h'\perp H',
  \end{equation}
  where $H$ and $H'$ are hyperbolic 
  forms over $(A,\s)$. Let $a\in \CP\cap A^\x$.
  By Proposition~\ref{Sylv} there exist  nonsingular
  quadratic forms $q', \vf,\psi$ over $F$ that all have nonzero signature at
  $P$ and 
  such that 
  $q'\ox h'$, $\vf\ox H$ and $\psi \ox H'$ are diagonal hermitian forms
  of the form $\pi\ox\qf{a}_\s \perp \nu \ox \qf{-a}_\s$, for some diagonal
  quadratic forms $\pi$ and $\nu$ with coefficients in $P^\x$. 
  Let 
  $\chi:=q_h\ox q'\ox \vf\ox \psi$.
  It then follows from
  \eqref{camargue} that
  \[
    \chi \ox h \perp \chi \ox H \simeq \chi \ox h' \perp
    \chi \ox H'.
  \] 
  Observe that by taking signatures at $P$, the form $\chi\ox H$ has as many 
  entries in $\CP\cap A^\x$ as in
  $-\CP\cap A^\x$ since it is still a hyperbolic form and thus has signature 
  zero. The same argument applies to $\chi \ox H'$.
  By definition of $q_h$, the form $\chi\ox h$ also has as many 
  entries in $\CP\cap A^\x$ as in $-\CP\cap A^\x$. It then follows from
  Lemma~\ref{samenr} that the form $\chi \ox h'$ has as many 
  entries in $\CP\cap A^\x$ as in $-\CP\cap A^\x$. Since $\sign_P \chi\not=0$,
  we conclude that $h'$
  satisfies Property~\eqref{Z} with $q_{h'}=\chi$. 
\end{proof}

\begin{defi}
  Denoting
  Witt classes with square brackets,
  we define
  \[N_\CP := \{[h] \in W(A,\s) \mid \text{Property~\eqref{Z} holds for } h\}\]
  and
  \[I_\CP:=\{[q]\in W(F) \mid \sign_Pq=0\},\]
  the ideal of $W(F)$ corresponding to the ordering $P$ (which is clearly
  generated by the classes in $W(F)$ of all elements of the form $\qf{1,-u}$ for
  $u \in P$).
\end{defi}
 
Recall again that by \cite[Lemma~3.6]{A-U-pos} there exists an invertible 
element
$a$ in $\CP$. It follows that the form $\qf{a,-a}_\s$ satisfies 
Property~\eqref{Z}, and in particular that $N_\CP\not=\varnothing$.

\begin{prop}\label{m-nt}
  The pair $(I_\CP, N_\CP)$ is an m-ideal 
  of $W(A,\s)$,  and $N_\CP \not =
  W(A,\s)$.
\end{prop}
\begin{proof}
  We have to check the following:
  \begin{enumerate}[label=(\arabic*)]
    \item $N_\CP+N_\CP \subseteq N_\CP$;
    \item $W(F) \cdot N_\CP \subseteq N_\CP$;
    \item $I_\CP \cdot W(A,\s) \subseteq N_\CP$;
    \item $N_\CP \not = W(A,\s)$.
  \end{enumerate}
  We do it in order.
  For the verification of (1) and (2) we fix  two hermitian
  forms $\vf$ and $\psi$
  over $(A,\s)$ that satisfy Property~\ref{Z}, so that $[\varphi], [\psi]
  \in N_\CP$.  Therefore, there are quadratic forms $q_\vf,q_\psi$ over $F$ 
  such that
  \[q_\vf \ox \vf \simeq \qf{a_1, \ldots, a_r}_\s \perp \qf{-b_1, \ldots,
    -b_r}_\s\]
    and
  \[q_\psi \ox \psi \simeq \qf{c_1, \ldots, c_s}_\s \perp \qf{-d_1, \ldots,
    -d_s}_\s\]
  for some $a_1, \ldots, a_r, b_1, \ldots, b_r, c_1, \ldots, c_s, d_1, \ldots,
  d_s \in \CP \cap A^\x$, and where $\sign_P q_\vf \not = 0$ and 
  $\sign_P q_\psi \not = 0$.
  \medskip

  (1) We show that $[\vf] + [\psi] = [\vf \perp \psi] \in N_\CP$ by showing
  that $\vf \perp \psi$ satisfies Property \eqref{Z}. We have
  \begin{multline}\label{what}
    (q_\vf \ox q_\psi) \ox (\vf \perp \psi) \simeq q_\psi 
    \ox  (\qf{a_1, \ldots, a_r}_\s \perp \qf{-b_1, \ldots,
    -b_r}_\s) \perp \\
     q_\vf \ox (\qf{c_1, \ldots, c_s}_\s \perp \qf{-d_1, \ldots, -d_s}_\s).
  \end{multline}
  Writing $q_\vf = q_+ \perp q_-$ and $q_\psi = q'_+ \perp q'_-$ with $q_+, 
  q'_+$
  positive definite at $P$ and $q_-, q'_-$ negative definite at $P$, we have 
  that the number of entries in $\CP\cap A^\x$ on the right-hand side 
  of \eqref{what}
    is
      \[(\dim q'_+)r + (\dim q'_-)r + (\dim q_+)s + (\dim q_-)s = (\dim 
      q_\psi)r +
      (\dim q_\vf)s,\]
  and that the number of entries in $-\CP\cap A^\x$ on the right-hand side 
    of \eqref{what} is
      \[(\dim q'_-)r + (\dim q'_+)r + (\dim q_-)s + (\dim q_+)s = (\dim 
      q_\psi)r +
      (\dim q_\vf)s.\]
  Both are equal, so $\vf \perp \psi$ satisfies Property \eqref{Z}.
  \medskip

(2) Since $W(F)$ is additively generated by classes of one-dimensional forms, it
    suffices to check that $[\qf{u}\ox \vf] \in N_\CP$ for every $u \in 
    F^\x$, which follows from the fact that the form $\qf{u} \ox \vf$ clearly
    satisfies Property \eqref{Z}.
    \medskip

(3) Let $\vf$ be a nonsingular hermitian form over $(A, \s)$. Since $I_\CP$ is
additively generated by the classes of the forms $\qf{1,-u}$ for $u \in P^\x$,
it suffices to check that $\qf{1,-u} \ox \vf$ satisfies Property \eqref{Z} for 
every
$u \in P^\x$. By Proposition~\ref{Sylv}
there is a nonsingular quadratic form $q_\vf$ over
$F$ such that $\sign_P q_\vf\not=0$ and
  \[
    q_\vf\ox \vf \simeq \qf{a_1,\ldots, a_r}_\s \perp \qf{-b_1,\ldots, -b_s}_\s
  \]
  for some $a_1,\ldots, a_r,-b_1,\ldots, -b_s \in \CP\cap A^\x $. Then
  \begin{multline*}
  q_\vf \ox (\qf{1,-u} \ox \vf) \simeq \qf{a_1, \ldots, a_r}_\s \perp 
  \qf{ub_1, \ldots, ub_s}_\s \perp\\ 
  \qf{-ua_1, \ldots, -ua_r}_\s \perp \qf{-b_1, \ldots,
  -b_s}_\s,
  \end{multline*}
  which shows that $\vf$ satisfies Property~\eqref{Z}.
    \medskip

(4)  Let $a \in \CP \cap A^\x$. We show that $[\qf{a}_\s] \not \in
    N_\CP$. Assume that it is not the case. Then there is a nonsingular 
    hermitian form
    $h$ over $(A,\s)$ such that $h$ satisfies Property~\eqref{Z} and $[h] =
    [\qf{a}_\s]$. It follows from Lemma~\ref{witteq} that $\qf{a}_\s$ also
    satisfies Property~\eqref{Z}, and thus  that
    \begin{equation}\label{four}
      q_{\qf{a}_\s} \ox \qf{a}_\s  \simeq \qf{a_1,
      \ldots, a_r}_\s \perp \qf{-b_1, \ldots, -b_r}_\s,
    \end{equation}
    with $a_1, \ldots, a_r, b_1, \ldots, b_r \in \CP\cap A^\x$.  
    We write 
    \[
      q_{\qf{a}_\s} \simeq
    \qf{u_1, \ldots, u_s} \perp \qf{-v_1, \ldots, -v_t} 
    \]
    with $u_1, \ldots,
    u_s, v_1, \ldots, v_t \in P^\x$. Since $\sign_P q \not = 0$ we have $s \not
    = t$. Equation \eqref{four} then becomes
    \[
      \qf{u_1a, \ldots, u_sa}_\s \perp
      \qf{-v_1a, \ldots, -v_ta}_\s 
      \simeq \qf{a_1, \ldots, a_r}_\s \perp \qf{-b_1, \ldots, -b_r}_\s.
    \] 
    By Lemma~\ref{samenr}, since the
    right-hand side has the same number of elements in $\CP$ as in $-\CP$, we
    must have $s=t$, contradiction.
\end{proof}

\begin{prop}\label{torsion-free}  The quotient
  $W(A,\s)/N_\CP$ is torsion-free and $(I_\CP, N_\CP)$ is a prime m-ideal.
\end{prop}
\begin{proof}
  Let $\ell[h] = [\ell \x h] \in N_\CP$ for some $\ell \in \N$, where $h$ is a
  nonsingular hermitian form over $(A,\s)$. By Lemma~\ref{witteq},  
  $\ell \x h$ satisfies
  Property \eqref{Z}. Then
  \[q_{\ell \x h} \ox (\ell \x h) \simeq \qf{a_1, \ldots, a_r}_\s \perp 
  \qf{-b_1, \ldots,
    -b_r}_\s\]
  for some $a_1, \ldots, a_r, b_1, \ldots, b_r \in \CP \cap A^\x$, and so,  
  clearly
  \[(\ell \x q_{\ell \x h}) \ox h \simeq \qf{a_1, \ldots, a_r}_\s \perp 
  \qf{-b_1, \ldots,
    -b_r}_\s,\]
  proving that $h$ satisfies Property \eqref{Z} and thus $[h] \in N_\CP$.

  We now prove the second statement: Assume that $[qh] \in N_\CP$ for some 
  $[q] \in
  W(F)$ and $[h] \in W(A,\s)$. Since $I_\CP$ is the kernel of $\sign_P : W(F)
  \rightarrow \Z$, there is $k \in \Z$ such that $[q] = k \bmod I_\CP$. Thus 
  (and
  using that $I_\CP \cdot W(A,\s) \subseteq N_\CP$) we obtain 
  $k[h] \in
  N_\CP$. It follows that $k=0$ (and thus $[q]\in I_\CP$), or that 
  $[h] \in N_\CP$
  by the first part.
\end{proof}

\begin{thm}\label{same-as-ker}  We have
  $(I_\CP,N_\CP) = (\ker \sign_P, \ker \sign^\mu_P)$ and $P\not\in 
  \Nil[A,\s]$.
\end{thm}
\begin{proof}
By definition, $I_\CP = \ker \sign_P \not \ni 2$. By \cite[Proposition
6.5]{A-U-prime} we obtain that $N_\CP = \ker \sign^\mu_P$. Therefore,
$P\not\in\Nil[A,\s]$ since $N_\CP\not=W(A,\s)$ by Proposition~\ref{m-nt}
and the equivalence in \eqref{sign7}.
\end{proof}

\begin{cor}\label{same-signature}
  Let $\CP$ be a positive cone on $(A,\s)$ over $P \in X_F$. Then, for every
  $a,b \in \CP \cap A^\x$, $\sign^\mu_P \qf{a}_\s = \sign^\mu_P \qf{b}_\s$.
\end{cor}
\begin{proof}
  The hermitian form $\qf{a,-b}_\s$ trivially satisfies Property \eqref{Z}.
  Therefore $[\qf{a,-b}_\s] \in N_\CP = \ker \sign^\mu_P$, so that
  $\sign^\mu_P \qf{a}_\s = \sign^\mu_P \qf{b}_\s$.
\end{proof}

\section{Description of positive cones and the
topology of $X_{(A,\s)}$}
\label{secdesc}

We use the previous results to describe positive cones in terms of 
$\sign^\mu_P$ and to establish some properties of the (Harrison) topology
$\CT_\s$ on the space of positive cones $X_{(A,\s)}$ of $(A,\s)$. Recall from
\cite[Section~9]{A-U-pos} that $\CT_\s$ is the topology generated by the sets
\[
  H_\s(a_1,\ldots, a_k):=\{\CP \in X_{(A,\s)} \mid a_1,\ldots, a_k \in \CP\},
\]
where $a_1,\ldots, a_k \in \Sym(A,\s)$.

Let $(D,\vt)$ be an $F$-division algebra with involution, and let $\eta$ be
a reference form for $(D,\vt)$.

\begin{prop}\label{M-max}
  Let $P \in X_F\sm \Nil[D,\vt]$. Then $\CM^\eta_P(D,\vt)$ is a positive cone on
  $(D,\vt)$ over $P$.
\end{prop}
\begin{proof}
  By \cite[Example~3.13]{A-U-pos} 
  it suffices to show that $\CM^\eta_P(D,\vt)$ is maximal. Let $\CP$ be a
  positive cone such that $\CM^\eta_P(D,\vt) \subseteq \CP$.  By 
   \cite[Lemma~3.16]{A-U-pos}, $\CP$ is over $P$, and by 
   Corollary~\ref{same-signature},
  $\CP = \CM^\eta_P(D,\vt)$.
\end{proof}

\begin{lemma}\label{adding}
  Let $\CP$ be a prepositive cone on $(D,\vt)$ over $P \in X_F$. Assume that
  $\sign^\eta_P \qf{b}_\vt > -m_P(D,\vt)$ for every $b \in \CP$. Let $a
  \in \Sym(D,\vt)\cap D^\x$ be such that $\sign^\eta_P \qf{a}_\vt = m_P(D,\vt)$.
  Then:
  \begin{enumerate}[label={\rm(\arabic*)}]
    \item $\CP[a]$  is a  prepositive cone on $(D,\vt)$ over $P$.
    \item For every $x \in \CP[a]$, $\sign^\eta_P \qf{x}_\vt > -m_P(D,\vt)$.
  \end{enumerate}
\end{lemma}
\begin{proof}
(1)  Properties (P1) to (P4) are straightforward to check for
  $\CP[a]$. We show that property (P5) holds, i.e., that $\CP[a]$ is proper.
  Assume $\CP[a]$ is not proper, 
  and let $b \in \CP[a] \cap -\CP[a]$, $b \not = 0$. Then
  there exist $p_1, p_2 \in \CP, k,r \in \N\cup\{0\}, u_i, v_j \in P$
  and $x_i, y_j\in D$ such that
  \[b = p_1 + \underbrace{\sum_{i=1}^k u_i\vt(x_i)ax_i}_{\mbox{$\alpha$}} 
  = -p_2 - \underbrace{\sum_{j=1}^r
  v_j\vt(y_j)ay_j}_{\mbox{$\beta$}}.\]
  Observe that at least one of $\alpha$ or $\beta$ is nonzero, 
  since $\CP$ is proper. Furthermore, $\alpha+\beta\not=0$.
  (Indeed, if $\alpha+\beta=0$, then $\alpha=-\beta\not=0$, contradicting 
  that $\CM^\eta_P(D,\vt)$ is proper  since $a\in \CM^\eta_P(D,\vt)$.)
  It follows that
  \[p_1 + p_2 = -\sum_{i=1}^k u_i\vt(x_i)ax_i - \sum_{j=1}^r v_j\vt(y_j)ay_j
  =-\alpha-\beta.\]
  The right-hand side is a nonzero
  sum of elements that are in $-\CM^\eta_P(D,\vt)$
  (since $a\in \CM^\eta_P(D,\vt)$), and thus
  belongs to $-\CM^\eta_P(D,\vt)$. The left-hand side, being in $\CP$, does not
  belong to $-\CM^\eta_P(D,\vt)$ by hypothesis, contradiction.

(2) Let $x = p + \underbrace{\sum_{i=1}^k u_i\vt(x_i)ax_i}_{\mbox{$\alpha$}}
  \in \CP[a]$. 
  Observe that $\alpha\in \CM^\eta_P(D,\vt)$. If $\alpha=0$, 
  then $x = p$ and $\sign^\eta_P \qf{x}_\vt > -m_P(D,\vt)$
  by hypothesis. If $\alpha \not= 0$, assume that $\sign^\eta_P \qf{x}_\vt =
  -m_P(D,\vt)$. We have $x-\alpha = p$. The left-hand side
  is a nonzero element of $-\CM^\eta_P(D,\vt)$ (the sum of two nonzero elements
  of a prepositive cone is nonzero by (P5)),  while the right-hand side is 
  not (by hypothesis), contradiction.
\end{proof}

Lemma~\ref{adding} leads to a proof of the following theorem, which was
stated in \cite{A-U-pos} and whose original proof relied on the incorrect
\cite[Lemma~5.5]{A-U-pos}.

\begin{thm}[{\cite[Proposition~7.1]{A-U-pos}}]\label{description}
  Let $\CP$ be a prepositive cone on $(D,\vt)$ over $P \in X_F$. Then
  $P\not\in \Nil[D,\vt]$ and either
  $\CP \subseteq \CM^\eta_P(D,\vt)$, or $\CP \subseteq -\CM^\eta_P(D,\vt)$.

  In particular $\CM^\eta_P(D,\vt)$ and $-\CM^\eta_P(D,\vt)$ are the only
    positive cones on $(D,\vt)$ over $P$, i.e., 
  \[
  X_{(D,\vt)} =
  \{-\CM^\eta_P(D,\vt), \CM^\eta_P(D,\vt) \mid P \in  X_F\sm \Nil[D,\vt] \}.
  \]
\end{thm}
\begin{proof}
Since $\CP$ is contained in a positive cone over $P$, we have $P\not\in \Nil[D,\vt]$ by 
Theorem~\ref{same-as-ker}.
  We now consider two cases.
  
  \emph{Case~1:} There is $c \in \CP$ such that $\sign^\eta_P \qf{c}_\vt =
  -m_P(D,\vt)$. Then by Lemma~\ref{same-signature}, $\CP \subseteq
  -\CM^\eta_P(D,\vt)$.

  \emph{Case~2:} For every $c \in \CP$, $\sign^\eta_P \qf{c}_\vt >
  -m_P(D,\vt)$. Then, using Lemma~\ref{adding}, we can add all elements of
  $\CM^\eta_P(D,\vt)$ to $\CP$ and we obtain in this way a  prepositive cone
  $\CQ$ containing both $\CP$ and $\CM^\eta_P(D,\vt)$.
  Since $\CM^\eta_P(D,\vt)$ is a maximal prepositive cone (cf. 
  Proposition~\ref{M-max})
  we obtain $\CQ = \CM^\eta_P(D,\vt)$ and thus $\CP \subseteq 
  \CM^\eta_P(D,\vt)$.
\end{proof}

Recall from \cite[(2.1)]{A-U-pos} (with $\ve=1$) 
that there is a hermitian Morita equivalence
\[
  g: \Herm(M_\ell(D), \vt^t)\to \Herm(D,\vt)   
\]
and that its inverse $g^{-1}$
sends any diagonal form $\qf{a_1,\ldots, a_\ell}_\vt$ to the form
$\qf{\diag(a_1,\ldots, a_\ell)}_{\vt^t}$. 

If $\CP$ is a prepositive cone on $(D,\vt)$ over $P\in X_F$, we define
\[
  \PSD_\ell(\CP):=\{B\in \Sym(M_\ell(D), \vt^t) \mid \forall X \in D^\ell\quad
    \vt(X)^tBX \in \CP\},
\]
cf. \cite[Section~4.1]{A-U-pos}.

\begin{lemma}[{\cite[Lemma~7.2]{A-U-pos}}]\label{CPSD}
  We have
  \[
  \CC_P(\CM^{g^{-1}(\eta)}_P(M_\ell(D),\vt^t)) = \PSD_\ell(\CM^\eta_P(D,\vt)).
  \] 
\end{lemma}

\begin{proof}
  Let $B \in \PSD_\ell(\CM_P^\eta(D,\vt))$. Then by \cite[Lemma~4.5]{A-U-pos}
  there
  is $G \in \GL_\ell(D)$ such that $\vt(G)^tBG = \diag(a_1, \ldots, a_\ell)$
  with $a_1, \ldots, a_\ell \in \CM^\eta_P(D,\vt)$, and we may assume that
  \[\vt(G)^tBG = \diag(a_1, \ldots, a_r,0,\ldots,0)\]
  with $a_1, \ldots, a_r \in \CM^\eta_P(D,\vt)\smz$. 
  
  Since 
  $a_i I_\ell \in
  \CM^{g^{-1}(\eta)}_P(M_\ell(D),\vt^t)$, it is now easy to represent 
  $\vt(G)^tBG$,
  and thus $B$, as an element of 
  $\CC_P(\CM^{g^{-1}(\eta)}_P(M_\ell(D),\vt^t))$, proving
  that 
  \[
  \PSD_\ell(\CM^\eta_P(D,\vt)) \subseteq 
  \CC_P(\CM^{g^{-1}(\eta)}_P(M_\ell(D),\vt^t)).
  \]
The equality follows since $\PSD_\ell(\CM^\eta_P(D,\vt))$ is a 
positive cone by Theorem~\ref{description}
and  \cite[Proposition~4.7]{A-U-pos}, and 
$\CC_P(\CM^{g^{-1}(\eta)}_P(M_\ell(D),\vt^t))$ is a prepositive cone by
Proposition~\ref{lausanne}.
\end{proof}

Assume now that $(D,\vt)$ is
the $F$-division algebra with involution 
that  is Morita
equivalent to $(A,\s)$. Observe that if $(A,\s)$ has at least
one positive cone, we may assume that the involutions $\vt$ and $\s$ are of
the same type, cf. \cite[Assumption on p.~8]{A-U-pos}, and 
it follows from \cite[Chapter~I, Theorem~9.3.5]{knus91}
that there is a hermitian Morita equivalence $\fm$ between 
the categories of hermitian forms
$\mathfrak{Herm}(A,\s)$ and $\mathfrak{Herm}(D,\vt)$. We let the
reference form $\eta$ be equal to   ${\mathfrak{m}(\mu)}$.

\begin{remark}\label{bleuarg}
  We note that \cite[Lemma~7.4]{A-U-pos} is now valid with its original proof,
  after replacing the reference to 
  \cite[Lemma~7.2]{A-U-pos} by a reference to Lemma~\ref{CPSD}, since they
  both prove the same statement.
\end{remark}

\begin{thm}[{\cite[Theorem~7.5]{A-U-pos}}]\label{positive=max}
  Let $\CP$ be a
  prepositive cone on $(A,\s)$ over $P \in X_F$. Then either
  \[\CP \subseteq \CC_P(\CM^\mu_P(A,\s)), \text{ or } \CP \subseteq
  -\CC_P(\CM^\mu_P(A,\s)).\]
  In particular
  \[X_{(A,\s)} = \{-\CC_P(\CM^\mu_P(A,\s)), \CC_P(\CM^\mu_P(A,\s)) \mid P \in  
  X_F\sm \Nil[A,\s] \},\]
  and
  for each $\CP \in X_{(A,\s)}$ there exists $\ve\in\{-1,1\}$ such that
  $\CP \cap A^\x = \ve \CM^\mu_P(A,\s) \setminus \{0\}$.
\end{thm}

\begin{proof}
  The original proof is still valid, but the references in it to the results in
  \cite{A-U-pos} stated after \cite[Section~5]{A-U-pos} need to be replaced by 
  the same
  results obtained in the current paper, as follow:
\[
  \begin{array}[b]{ll}
    \text{\underline{Original proof in \cite{A-U-pos}:}}\ & 
    \text{\underline{Corresponding statement in 
    this paper:}} \\[5pt]
    \text{Proposition~6.6} & \text{Theorem~\ref{same-as-ker}} \\
    \text{Proposition~7.1} & \text{Theorem~\ref{description}} \\
    \text{Lemma~7.2} & \text{Lemma~\ref{CPSD}} \\
    \text{Proposition~7.1} & \text{Theorem~\ref{description}}\\
    \text{Lemma~7.4} & \text{cf. Remark~\ref{bleuarg}}
  \end{array}\qedhere
\]
\end{proof}

\begin{cor}\label{PUD}
  Let $P\in X_F$. Then
  \[
    \CC_P(\CM^\mu_P(A,\s)) = \bigcup \{D_{(A,\s)} \qf{a_1, \ldots, a_k}_\s 
    \mid k \in \N, \ a_1, \ldots, a_k \in \CM^\mu_P(A,\s)\}.
  \]
  In particular, if $\CP$ is a positive cone on $(A,\s)$ over $P \in X_F$ such 
  that $\CP \cap
  A^\x = \CM^\mu_P(A,\s) \setminus \{0\}$, then
  \[\CP = \bigcup \{D_{(A,\s)} \qf{a_1, \ldots, a_k}_\s \mid k \in \N, \ a_1, 
  \ldots, a_k \in \CM^\mu_P(A,\s)\}.\]
\end{cor}
\begin{proof}
  The first statement is clear by definition of $\CC_P$ and since 
  $\CM^\mu_P(A,\s)$ is closed under multiplication by elements of $P\smz$.
  The second statement follows immediately from Theorem~\ref{positive=max}.
\end{proof}

\begin{prop}[{\cite[Proposition~9.11(3)]{A-U-pos}}]
  The topology $\CT_\s$ is compact (by which me mean quasicompact).
\end{prop}
\begin{proof}
  A positive cone $\CP$ is a subset of $\Sym(A,\s)$, so can be identified with a
  map from $\Sym(A,\s)$ to $\{0,1\}$ (with $\CP(a) = 1$ iff $a \in \CP$).
  Thus we can view  $X_{(A,\s)}$ as a subset of $Z := \{0,1\}^{\Sym(A,\s)}$ and 
  the topology
  $\CT_\s$ as the topology induced by the product topology $T$ on
  $Z$ of the discrete topology on $\{0,1\}$. Since $T$ is compact, it suffices 
  to show that
  $X_{(A,\s)}$ is a closed subset of $Z$. We slightly
  reformulate the prepositive cone properties 
  (P1), (P4) and (P5) in order to make it
  easier to check them:
  \begin{itemize}
    \item[(P1)] $0 \in \CP$.
    \item[(P4)] $\forall u \in F \ u \in \CP_F \vee -u \in \CP_F$. (This
      reformulation is equivalent to the original (P4)
      since $\CP_F$ is always a preordering, so we only
      need to check that it is total.)
    \item[(P5)] $\forall a \in \Sym(A,\s) \setminus \{0\} \ \neg(a \in \CP
      \wedge -a \in \CP)$.
  \end{itemize}

  We now show that the subset $S_i$ of $Z$ of subsets
  satisfying property (P$i$) is closed, for $i=1, \ldots, 5$. The result follows
  since $X_{(A,\s)} = S_1 \cap \cdots \cap S_5$.

  \[S_1 = \{\CP \in Z \mid \CP(0) = 1\},\]
  which is closed in $T$.

  \begin{align*}
    S_2 &= \{\CP \in Z \mid \forall a,b \in \Sym(A,\s) \ a,b \in \CP \Rightarrow
      a+b \in \CP\} \\
      &= \{\CP \in Z \mid \forall a,b \in \Sym(a,\s) \ \neg(a,b \in \CP) \vee
      a+b \in \CP\} \\
      &= \{\CP \in Z \mid \forall a,b \in \Sym(a,\s) \ a \not \in \CP \vee b
      \not \in \CP \vee a+b \in \CP\} \\
      &= \bigcap_{a,b \in \Sym(A,\s)} \{\CP \in Z \mid \CP(a) = 0 \vee \CP(b) =
      0 \vee \CP(a+b) = 1\},
  \end{align*}
  which is an intersection of closed sets in $T$, and therefore closed.

  \begin{align*}
    S_3 &= \{\CP \in Z \mid \forall a \in \Sym(A,\s) \forall x \in A \ a \in \CP
      \Rightarrow \s(x)ax \in \CP\} \\
      &= \{\CP \in Z \mid  \forall a \in \Sym(A,\s) \forall x \in A \ a \not \in
      \CP \vee \s(x)ax \in \CP\} \\
      &= \bigcap_{a \in \Sym(A,\s),\ x \in A} \{\CP \in Z \mid \CP(a) = 0 \vee
      \CP(\s(x)ax) = 1\}, 
  \end{align*}
  which is an intersection of closed sets in $T$, and therefore closed.

  \begin{align*}
    S_4 &= \{\CP \in Z \mid \forall u \in F\ u \in \CP_F \vee -u \in \CP_F\} \\
        &= \{\CP \in Z \mid \forall u \in F\ (\forall a \in \CP \ ua \in \CP)
        \vee (\forall a \in \CP \ -ua \in \CP)\} \\
        &= \bigcap_{u \in F} \{\CP \in Z \mid (\forall a \in \CP \ ua \in \CP)
        \vee (\forall a \in \CP \ -ua \in \CP)\} \\
        &= \bigcap_{u \in F} \{\CP \in Z \mid \forall a \in \CP \ ua \in \CP\}
        \cup \{\CP \in Z \mid \forall a \in \CP \ -ua \in \CP\} \\
        &= \bigcap_{u \in F} \Bigl( \{\CP \in Z \mid 
        \forall a \in \Sym(A,\s)\ a \not
        \in \CP \vee ua \in \CP\} \cup \\
        & \qquad\qquad \{\CP \in Z \mid \forall a \in \Sym(A,\s)\ a \not \in 
        \CP \vee
        -ua \in \CP\} \Bigr)\\
        &= \bigcap_{u \in F} \Bigl(\bigcap_{a \in \Sym(A,\s)} \{\CP \in Z \mid
        \CP(a) = 0 \vee \CP(ua) = 1\}\Bigr) \cup \\
        & \qquad \qquad \Bigl(\bigcap_{a \in \Sym(A,\s)} \{\CP \in Z \mid 
        \CP(a) = 0 \vee
        \CP(-ua)=1\}\Bigr),
  \end{align*}
  which is a closed set in $T$.

  \begin{align*}
    S_5 &= \{\CP \in Z \mid \forall a \in \Sym(A,\s) \setminus \{0\} \ \neg(a
    \in \CP \wedge -a \in \CP)\} \\
        &= \{\CP \in Z \mid \forall a \in \Sym(A,\s) \setminus \{0\} \ 
        a \not \in \CP \vee -a \not \in \CP)\} \\
        &= \bigcap_{a \in \Sym(A,\s) \setminus \{0\}} \{\CP \in Z \mid 
        \CP(a)=0 \vee   \CP(-a)=0)\},
  \end{align*}
  which is closed in $T$.
\end{proof}

\begin{prop}[{\cite[Proposition~9.7(2)]{A-U-pos}}]
  The map 
  \[
    \pi : X_{(A,\s)} \rightarrow X_F,\ \CP \mapsto \CP_F
  \]  
  is 
  continuous, where $X_F$ is equipped with the usual Harrison topology.
\end{prop}
\begin{proof}
  The proof is the same as the proof of
  \cite[Proposition~9.7(2)]{A-U-pos}, except that we use an infinite union
  instead of a finite one in the final part: 
  Let $u \in F \setminus \{0\}$. We show that
  $\pi^{-1}(H(u))$ is open. By definition,
  \[\pi^{-1}(H(u)) = \{\CP \in X_{(A,\s)} \mid u \in \CP_F\}.\]
  Observe that if $c \in \CP \setminus \{0\}$, then $u \in \CP_F$ if and only
  if $uc \in \CP$ (indeed, $u \in \CP_F$ or $u \in -\CP_F$, and 
  only one of them occurs by (P5);
  the first case
  corresponds to $uc \in \CP$). Therefore, $\CP \in \pi^{-1}(H(u))$ if and only
  if there is $c \in \Sym(A,\s) \setminus \{0\}$ such that $c \in \CP$ and $uc
  \in \CP$. Thus
  \begin{align*}
    \pi^{-1} (H(u)) &= \bigcup_{c \in \Sym(A,\s) \setminus \{0\}} (H_\s(c) \cap
      H_\s(uc)), 
  \end{align*}
  which is open in $\CT_\s$.
\end{proof}

\begin{cor}\label{map-pi}
  The map $\pi$ is closed.
\end{cor}
\begin{proof}
  Since $X_{(A,\s)}$ is compact, $X_F$ is Hausdorff, and $\pi$ is continuous,
  the map $\pi$ is necessarily closed.   
\end{proof}

\begin{prop}[{\cite[Proposition~9.7(1)]{A-U-pos}}]\label{pi-cont}
  The map $\pi$ is open.
\end{prop}
\begin{proof}
  We show that $\pi(H_\s(a_1, \ldots, a_k))$ is open for all $k \in \N$ and
  $a_1, \ldots, a_k \in \Sym(A,\s)$. Letting $\wt X_F :=X_F\sm \Nil[A,\s]$,
  we first observe that
  \begin{equation}\label{xmas}
    X_F \setminus \pi(H_\s(a_1, \ldots, a_k))=\Nil[A,\s] \cup (\wt X_F \setminus
    \pi(H_\s(a_1, \ldots, a_k)))
  \end{equation}
  since $\im\pi \subseteq \wt X_F$, and we show
  \begin{multline}\label{eq-pi}
    \wt X_F \setminus \pi (H_\s(a_1, \ldots, a_k)) = \pi\bigl(X_{(A,\s)} 
    \setminus (H_\s(a_1,
    \ldots, a_k)\\ \cup H_\s(-a_1, \ldots, -a_k))\bigr).
  \end{multline}
  
  ``$\subseteq$'': Let $P \in \wt X_F \setminus \pi(H_\s(a_1, \ldots, a_k))$ and
  let $\CP$ be a positive cone over $P$, so that
  $P=\pi(\CP)=\pi(-\CP)$. We want to show that $\CP \not \in 
  H_\s(a_1, \ldots, a_k) \cup H_\s(-a_1, \ldots, -a_k)$. If $\CP \in H_\s(a_1,
  \ldots, a_k)$, then $P \in \pi(H_\s(a_1, \ldots, a_k))$, contradiction. If 
  $\CP
  \in H_\s(-a_1, \ldots, -a_k)$, then $P = \pi(-\CP) \in \pi(H_\s(a_1, \ldots,
  a_k))$, contradiction again.

  ``$\supseteq$'': Let $\CP \in X_{(A,\s)} \setminus (H_\s(a_1, \ldots, a_k) 
  \cup
  H_\s(-a_1, \ldots, -a_k))$ be over $P \in X_F$, so that $\pi(\CP) = P$. If $P
  \in \pi(H_\s(a_1, \ldots, a_k))$, then there is a positive cone
  $\CQ$ over $P$ such that $\CQ
  \in H_\s(a_1, \ldots, a_k)$. In particular $\CQ = \CP$ or $\CQ = -\CP$ by 
  Theorem~\ref{positive=max} since $\CQ$, $\CP$ and $-\CP$ are all over $P$.
  Then $\CP \in H_\s(a_1, \ldots, a_k)$
  in the first case, and $\CP \in H_\s(-a_1, \ldots, -a_k)$ in the second case,
  which are both contradictions.

  The right-hand side of \eqref{eq-pi} is $\pi$ of a closed set, so is closed by
  Corollary~\ref{map-pi}. 
  Therefore the left-hand side is closed, which shows that the set
  $\pi(H_\s(a_1, \ldots, a_k))$ is open in $X_F$
  by \eqref{xmas} and since $\Nil[A,\s]$ is clopen by
  \cite[Corollary~6.5]{A-U-kneb}.
\end{proof}

\section{Maximum signatures and extension of positive cones}
\label{secmax}

Let $P\in X_F\sm \Nil[A,\s]$. 
Recall from Section~\ref{csa-sign}
that 
\begin{equation}\label{nP}
  (A\ox_F F_P,\s\ox\id)\cong (M_{n_P}(D_P), \Int(\Phi_P) \circ 
  {\vt_P}^t),
\end{equation}
where $D_P\in \{F_P, F_P(\sqrt{-1}), (-1,-1)_{F_P}\}$, $\vt_P$ is the canonical
involution on $D_P$, and 
$\Phi_P \in \Sym_\ve (M_{n_P}(D_P), {\vt_P}^t) \cap  M_{n_P}(D_P)^\x$.
By Proposition~\ref{lausanne} there exists a positive
cone on $(A,\s)$ over $P$. Therefore we may assume that $\ve=1$ by
\cite[Corollary~3.8]{A-U-pos}. 

We denote the integer $n_P$ that occurs in \eqref{nP} by $n_P(A,\s)$ if we 
want to emphasize the dependence on $(A,\s)$.

\begin{prop}\label{dense-rc}
  Assume that $F$ is dense in $F_P$ (for the
  topology induced by the ordering $P$). Then $m_P(A,\s) = n_P(A,\s)$.
\end{prop}
\begin{proof}
  Observe that by the definition of signatures, 
  $m_P(A,\s) \le n_P(A,\s)$, cf. \cite[Proposition~4.4(iii)]{A-U-PS}.
  For ease of notation we assume that $(A\ox_F F_P,\s\ox\id)=
  (M_{n_P}(D_P), \Int(\Phi_P) \circ {\vt_P}^t)$. 
  Let
  \begin{multline*}
    \PD_{n_P}(D_P,\vt_P):=\{B \in \Sym( M_{n_P}(D_P), {\vt_P}^t ) \mid
    \vt_P(X)^t BX > 0 \text{ in }F_P,\\\text{for every }X \in (D_P)^{n_P}\smz\}.
  \end{multline*}

  Then $\Phi_P\cdot\PD_{n_P}(D_P,\vt_P)$ is an open subset of 
  $ \Sym (M_{n_P}(D_P), {\vt_P}^t)$ by \cite[Lemma~1.21]{A-U-Az-PLG}.
  Since $F$ is dense in 
  $F_P$, $A\ox 1_{F_P}$ is dense in $A\ox_F F_P$, and
  there is $a \in A$ such that $a \ox 1_{F_P}
  \in \Phi_P\cdot\PD_{n_P}(D_P,\vt_P)$. 
  
  Let $\mathfrak{s}_P$ denote the hermitian Morita equivalence
  \[
    \Herm (M_{n_P}(D_P), \Int(\Phi_P) \circ \vt^t) \to \Herm (M_{n_P}(D_P), 
    \vt^t)
  \]
  given by the scaling map $h\mapsto \Phi_P^{-1}h$.   
  Denoting the unique ordering on $F_P$ by $\tP$, we then have  
  \[\sign^\mu_P \qf{a}_\s = \sign^{\mu\ox 1}_\tP 
  \qf{a\ox 1_{F_P}}_{\s \ox \id} =
  \sign^{\mathfrak{s}_P (\mu\ox 1)}_\tP \qf{\Phi_P^{-1} \cdot 
  a\ox 1}_{{\vt_P}^t} = \pm
  n_P(A,\s),\]
  where the second equality follows from \cite[Theorem~4.2]{A-U-prime},
  the final equality holds 
  since we are computing 
  the signature of a positive definite $n_P\x n_P$ matrix, and
  where the $\pm$ is due to the reference form $\mathfrak{s}_P (\mu\ox 1)$, 
  which may induce a sign change in the result. Replacing $a$ by $-a$ if
  necessary, the conclusion follows.
\end{proof}

\begin{prop}\label{arch-dense}
  Let $F$ be a finitely generated extension of $\Q$. Then the set of archimedean
  orderings on $F$ is dense in $X_F$.
\end{prop}
\begin{proof}
  Since $F$ is finitely generated over $\Q$, we can write $F=\Q(S)$ for some
  finite set $S$. By \cite[Chapter~IV, Theorem~8.6]{G07},
  $F$ has a transcendence basis $\{X_1, \ldots, X_n\}$ over $\Q$
  which is included in $S$. Thus, by the primitive element theorem,
  $F = \Q(X_1, \ldots, X_n)(\alpha)$ with 
  $\alpha$ algebraic over $\Q(X_1, \ldots, X_n)$.
  
  Let $m_{\bar X}(X)$ be the minimal polynomial of $\alpha$ over
  $\Q(\bar X)$, where $\bar X = (X_1,\ldots,X_n)$. Let $P \in X_F$ and let 
  $U$ be a
  basic open set containing $P$. The set $U$ is of the form
  \[\{Q \in X_F \mid\  g_1(\bar X, \alpha) >_Q 0, \ldots, 
  g_r(\bar X, \alpha) >_Q 0\},\]
  where the rational functions  $g_1,\ldots, g_r$ can be chosen to be 
  polynomials in $\Q[X_1,\ldots,X_n, \alpha]$.
  
  \begin{claim}
  There are $y_1, \ldots, y_n, \beta \in \R$ such that:
  \begin{enumerate}[label={\rm(\arabic*)}]
    \item $\{y_1, \ldots, y_n\}$ is algebraically independent over $\Q$;

    \item $g_j(y_1, \ldots, y_n, \beta) > 0$ for $j=1, \ldots, r$ (the ordering
      is the one from $\R$);
    \item $\beta$ is a root of $m_{\bar y}(X)$, where $\bar y = (y_1, \ldots,
      y_n)$.
  \end{enumerate}
  \end{claim}
  
  We will prove the Claim in the course of the next two lemmas, 
  but we use it now.
  The map $\lambda : F = \Q(X_1, \ldots, X_n, \alpha) \rightarrow \R$ defined by
  $\lambda(X_i) = y_i$ and $\lambda(\alpha) = \beta$ is a morphism of fields,  
  and $Q := \lambda^{-1}(\R_{\geq 0})$ is an ordering on $F$ such that 
  $Q \in U$. Writing $F' := \im \lambda =  \Q(y_1, \ldots, y_m, \beta)$, the 
  map $\lambda$
  gives an isomorphism of ordered fields $(F, Q) \cong (F', F' \cap 
  \R_{\geq 0})$,
  which is an ordered subfield of $(\R, \R_{\geq 0})$ and therefore archimedean.
\end{proof}

\begin{lemma}\label{S-open}
  With notation as in the proof of Proposition~\ref{arch-dense},
  let
  \begin{multline*}
    S := \{\bar x := (x_1, \ldots, x_n) \in \R^n \mid  m_{\bar x}(X) 
    \text{ has a root $\alpha$ such
      that } \\
      g_i(\bar x, \alpha) > 0 \text{ for } i=1, \ldots, r\}.
  \end{multline*}
  Then $S$ contains a non-empty open subset of $\R^n$.
\end{lemma}
\begin{proof}
  For $e \in \{1,2\}^r$, let $g^e := g_1^{e_1} \cdots g_r^{e_r}$ and 
  \[N_e(\bar x) := \sum_{c \in \R:\ m_{\bar x}(c) = 0} \sgn g^e(\bar x, c),\]
  where $\sgn$ denotes the sign function.
  By \cite[Proposition~1.3.36]{S24} we have
  \[\bigl|\{c \in \R \mid m_{\bar x}(c) = 0,\ g_1(\bar x,c) >0, \ldots,\ 
  g_r(\bar x,c) > 0\}\bigr| = 
  2^{-r} \sum_{e \in \{1,2\}^r} N_e(\bar x),\]
  for all $\bar x\in\R^n$,
  and thus, 
  \[S = \Bigl\{\bar x \in \R^n \Bigmid \sum_{e \in \{1,2\}^r} N_e(\bar x) \ge 
  2^r\Bigr\}.\]
  Using this description of $S$, we show that $S$
  contains an open subset. Listing all the possible values of $N_e(\bar x)$,  
  for $e \in \{1,2\}^r$, 
  whose sum gives a result greater than or equal to $2^r$, we see
  that $S$ can be expressed as a finite union of finite intersections of
  sets of the form
  \[S_{e,\ell} := \{\bar x \in \R^n \mid N_e(\bar x) =
  \ell\},\]
  for some $e \in \{1,2\}^r$ and $\ell \in \N \cup \{0\}$. 
  More precisely, we write
  \[S = \bigcup_{i \in I} \bigcap_{(e,\ell) \in E_i} S_{e, \ell},\]
  where $I$ and each $E_i$ are finite. We will show that each set $S_{e, \ell}
  \cap Z$ is an open subset of $\R^n$ 
  (for $i \in I$ and $(e,\ell) \in E_i$), where $Z$ is an open set that will be 
  defined
  below, thus proving that $S \cap Z$ is open. We will then show that 
  $S \cap Z$ is non-empty. 
  
  We
  first show that $S_{e, \ell} \cap Z$ is open (this will help us decide what
  $Z$ should be).
  For $i \in I$ and $(e, \ell) \in E_i$, let $(f_{e,0}, \ldots, f_{e,t_e})
  \in \Q(x_1, \ldots, x_n)[X]$ be the Sturm sequence of $m$ and $g^e$, and
  \begin{itemize}
    \item let $p_e:=(p_{e,1}, \ldots, p_{e,t_e})$ be the sequence of 
    coefficients 
    of the highest
      degree terms of $(f_{e,1}(X), \ldots, f_{e,t_e}(X))$;
    \item let $\tilde p_e:= (\tilde p_{e,1}, \ldots, \tilde p_{e,t_e})$ be the 
    sequence of coefficients of
      the highest degree terms of $(f_{e,1}(-X), \ldots, f_{e,t_e}(-X))$.
  \end{itemize}

  Let $v(m_{\bar x},g^e)$ be the number of sign changes in the sequence $p_e$
  and $\tilde v(m_{\bar x},g^e)$ the number of sign changes in the sequence
  $\tilde p_e$. By \cite[Corollary~1.2.12]{BCR}, $N_e(\bar x)$ is equal to
  $\tilde v(m_{\bar x},g^e) - v(m_{\bar x},g^e)$, and thus
  \[S_{e, \ell} = \{\bar x   \in \R^n \mid \tilde v(m_{\bar x},g^e) - v(m_{\bar 
  x},g^e)
  = \ell\}.\]
  Listing all the possible values of $\tilde v(m_{\bar x},g^e)$ and $v(m_{\bar
  x},g^e)$ whose difference gives $\ell$, we see that $S_{e, \ell}$ is a finite
  union of sets of the form
  \[T_1 := \{\bar x  \in \R^n \mid \tilde v(m_{\bar x},g^e) = k_1 \wedge 
  v(m_{\bar x}, g^e) = k_2\},\]
  with $k_1, k_2 \in \N \cup \{0\}$. We define
  \[
    Z := \{\bar x  \in \R^n \mid  \text{all } p_{e,j} \text{ and all }
      \tilde p_{e,j} \text{ are} \not = 0, 
     \text{ for all } i \in I \text{ and } (e, \ell) \in E_i\}.
  \] 
  The idea behind the introduction of $Z$ is that, if $\bar x  \in
  Z$, then the various sequences $p_e$ and $\tilde p_e$ never contain a zero,
  which makes it easier to describe their sign changes.

  The set $Z$ is clearly open. We show that $T_1 \cap Z$ is open, and for
  this it suffices to show that if
  \[T_2 :=  \{\bar x  \in \R^n \mid v(m_{\bar x},g^e) = k\}\]
  with $k\in\N\cup\{0\}$,
  then $T_2 \cap Z$ is open (since the other condition, on $\tilde v$, can be
  checked in the same way).

  Since, for $\bar x \in Z$, the coefficients of $p_e$ are all
  non-zero, we can express that $\bar x $ is in $T_2 \cap Z$ by
  listing all the configurations $p_{e,i} < 0$ / $p_{e,i} >0$ that enumerate 
  all the
  different ways in which $k$ sign changes can be obtained in the 
  sequence $(p_{e,1},\ldots, p_{e,t_e})$. But this
  clearly defines an open subset of $\R^n$.
  \medskip

  We now check that $S\cap Z$ is non-empty using some basic model theory.  Let
  $\vf(\omega_1, \ldots, \omega_n)$ be the following
  first-order formula in the language
  of rings:
  \[\exists \alpha \ m_{\bar \omega}(\alpha) = 0 \wedge \bigwedge_{i=1}^r
  g_i(\omega_1, \ldots, \omega_n, \alpha) > 0.\]
  By choice of $P \in X_F$ in the proof of Proposition~\ref{arch-dense},  we
  have
  \[(F,P) \models \vf(X_1, \ldots, X_n),\]
  Since $\{X_1, \ldots, X_n\}$ is a transcendence basis of $F$ over $\Q$, if 
  \[\{r_j(x_1, \ldots, x_n)\}_{j \in J}\] 
  is the finite list of all polynomials 
  that appear in the
  definition of the set $Z$ above (the polynomials 
  $p_{e,j}$ and $\tilde p_{e,j}$), we have
    \[(F,P) \models \vf(X_1, \ldots, X_n) \wedge \bigwedge_{j \in J} r_j(X_1,
  \ldots, X_n) \not = 0.\]
  Since the formula $\vf$ is existential, it follows that
    \[(F_P, \tP) \models \vf(X_1, \ldots, X_n) \wedge \bigwedge_{j \in J} 
      r_j(X_1,
        \ldots, X_n) \not = 0.\]
  In particular,
    \[(F_P, \tP) \models \exists x_1, \ldots, x_n \ \vf(x_1, \ldots, x_n)
    \wedge \bigwedge_{j \in J} r_j(x_1,
  \ldots, x_n) \not = 0,\]
  and thus, by Tarski's transfer principle (cf. \cite[Corollary~11.5.4]{mar08}),
  \[(\R, \R_{\geq 0}) \models \exists x_1, \ldots, x_n \ \vf(x_1, \ldots, x_n)
  \wedge \bigwedge_{j \in J} r_j(x_1,
  \ldots, x_n) \not = 0.\]
  The open set $S\cap Z$ is thus non-empty.
\end{proof}

\begin{lemma}
  The set $S$ defined in Lemma~\ref{S-open}
  contains a tuple $(x_1, \ldots, x_n)$ of elements that
  is algebraically independent over $\Q$.
  In particular, the Claim in the proof of Proposition~\ref{arch-dense} 
  is verified.
\end{lemma}
\begin{proof}
  By Lemma~\ref{S-open}, there exist non-empty
  open intervals $I_1, \ldots, I_n$ 
  of $\R$ such that
  $I_1 \x \cdots \x I_n \subseteq S$. 
  Let $x_1 \in I_1$ be transcendental over $\Q$. Assume that we have
  obtained $x_1 \in I_1, \ldots, x_k \in I_k$ with $(x_1, \ldots, x_k)$
  algebraically independent over $\Q$, for some $k < n$. Since $I_{k+1}$ is
  uncountable, there is $x_{k+1} \in I_{k+1}$ such that $(x_1, \ldots, x_{k+1})$
  is algebraically independent over $\Q$, and we conclude by induction.
\end{proof}

Let $L$ be the language of rings $L_r$ together with
$\{\ul{\s},\ul{F},\ul{P},\ul{\CP}\}$, where $\ul{\s}$ is a new
unary function symbol
and $\ul{F},\ul{P},\ul{\CP}$ are new unary relation symbols.

If $\CP$ is a prepositive cone on $(A,\s)$ over $P \in X_F$, we denote by $\CA$
the $L$-structure consisting of the algebra $A$ with the obvious
interpretation of the symbols of $L$: $\ul{\s}$ is interpreted by $\s$,
$\ul{F}$ by $F$, $\ul{P}$ by $P$, and $\ul{\CP}$ by $\CP$.

\begin{lemma}\label{elementary}
  Let $P \in X_F$ and let $\CP$ be a positive cone on $(A,\s)$ over $P$.
  Assume that $P$ belongs to the closure of the set of
  archimedean orderings of $F$. Then there is an elementary extension $N$ of 
  $F$ 
  (in the language $L_r$)
  and an ordering
  $Q$ on $N$ extending $P$ such that
  \begin{enumerate}[label={\rm(\arabic*)}]
    \item $(N,Q)$ is dense in its real closure;
    \item\label{elementary-2} There is a positive cone $\CQ$ on $(A \ox_F N, \s 
    \ox \id)$ over $Q$
      such that $\CP \ox 1 \subseteq \CQ$.
  \end{enumerate}
\end{lemma}
\begin{proof}
  Let $\Phi$ be the collection of formulas (without parameters) in the language
  of ordered fields expressing that an ordered field is dense in its real
  closure (cf. \cite[Remark~4.4]{KKL} and the references mentioned there), with
  quantifiers relativized to $\ul{F}$ (the formulas will be interpreted in an
  $L$-structure such as $\CA$, in which case we want them to be true if and
  only if they  are true in the interpretation of $\ul{F}$).

  Let $\Delta(F)$ be the complete theory (with parameters) of $F$ in the
  language $L_r$ (i.e., the set of all first-order $L_r$-formulas with 
  parameters in
  $F$ that are true in $F$), and where the quantifiers are relativized to
  $\ul{F}$.

  Fix an $F$-basis $\{e_1, \ldots, e_m\}$ of $A$. We define the structure
  constants $f_{ijk} \in F$ of $A$ with respect to this basis by:
  \[e_i e_j = \sum_{k=1}^m f_{ijk} e_k, \text{ for } 1 \le i,j \le m,\]
  and the constants $f_{\s ik}$ defining $\s$ by
  \[\s(e_i) = \sum_{k=1}^m f_{\s ik} e_k, \text{ for } 1\leq i \leq m.\]

  We consider the set of $L$-formulas:
  \begin{align*}
    \Omega := & \Delta(F) \cup \Phi \cup \{\ul{P} \text{ ordering on } \ul{F}\}
    \cup \{\ul{\CP} \text{ is a prepositive cone
    over } \ul{P}\} \\
    & \cup \{\{e_1, \ldots, e_m\} \text{ is a basis over } \ul{F}\} \cup 
    \Bigl\{ e_i e_j = \sum_{r=1}^m f_{ijk} e_k \Bigmid 1 \le i,j \le m\Bigr\} \\
    & \cup \Bigl\{\ul{\s}(e_i) = \sum_{k=1}^m f_{\s ik} e_k \Bigmid 1\leq i 
    \leq m \Bigr\} \\
    & \cup \{u \in \ul{P} \mid\ u \in P\} \cup \{a \in \ul{\CP} \mid a
    \in \CP\}.
  \end{align*}
  Let $S$ be a finite subset of $\Omega$. Thus, $S$ is included in
  \begin{align*}
    S' := & \Delta(F) \cup \Phi \cup \{\ul{P} \text{ ordering on }
    \ul{F}\} \cup \{\ul{\CP} \text{ is a prepositive cone over } \ul{P}\} \\
    & \cup \{\{e_1, \ldots, e_m\} \text{ is a basis over } \ul{F}\} \cup 
    \Bigl\{ e_i e_j = \sum_{r=1}^m f_{ijk} e_k \Bigmid 1 \le i,j \le m\Bigr\} \\
    & \cup \Bigl\{\ul{\s}(e_i) = \sum_{k=1}^m f_{\s ik} e_k \Bigmid 1\leq i 
    \leq m\Bigr\} \\
    & \cup \{u_1 \in \ul{P}, \ldots, u_k \in \ul{P}\} \cup \{a_1 \in
    \ul{\CP}, \ldots, a_\ell \in \ul{\CP}\},
  \end{align*}
  for some $u_1, \ldots, u_k \in P$ and $a_1, \ldots, a_\ell \in \CP$.

  Consider the open set $H(u_1, \ldots, u_k)$ of $X_F$ and the open set
  $H_\s(a_1, \ldots, a_\ell)$ of $X_{(A,\s)}$. Clearly, $P \in H(u_1, \ldots,
  u_k)$ and $\CP \in H_\s(a_1, \ldots, a_\ell)$. Recall that the map $\pi :
  X_{(A,\s)} \rightarrow X_F$, $\pi(\CQ) = \CQ_F$ is open by
  Proposition~\ref{pi-cont}. 
  Therefore $H(u_1, \ldots, u_k) \cap \pi(H_\s(a_1, \ldots,
  a_\ell))$ is an open subset of $X_F$ 
  containing $P$, and thus contains an ordering $P'$ 
  such that $(F,P')$ is archimedean by 
  hypothesis. Since $P' \in
  \pi(H_\s(a_1, \ldots, a_\ell))$, there is a positive cone $\CP'$ on $(A,\s)$
  over $P'$ such that $a_1, \ldots, a_\ell \in \CP'$.

  In particular, the $L$-structure $(A,\s,F,P',\CP')$ is a model of $S'$.
  (Recall that an archimedean ordered field is dense in its real closure.
  This follows directly from \cite[Theorem~1.1.5]{PD}.)
  Therefore every finite subset of $\Omega$ has a model, so that $\Omega$ has a
  model $\CB = (B, \tau, N, Q, \CS)$ by the compactness theorem. 

  By construction, $N$ is an elementary extension of $F$, $P \subseteq Q$, 
  $\CP\subseteq\CS$,  and
  $(N,Q)$ is dense in its real closure. 
  
  To prove statement \ref{elementary-2}, we first check that $(B, \tau) \cong 
  (A \ox_F N,
  \s \ox \id_N)$:  the structure constants of $B$ with respect
  to $\{e_1, \ldots, e_m\}$ are by construction 
  the same as those of $A$, and therefore as those
  of $A \ox_F N$. Thus, since both $B$ and $A \ox_F N$ are algebras over $N$ of 
  the
  same dimension, they are
  isomorphic, and the isomorphism is induced by $e_i \mapsto e_i \ox 1$
  for $1\leq i\leq m$.  Similarly, the
  $N$-linear maps $\tau$ and $\s \ox \id_N$ have the same matrix 
  with respect to $\{e_1,
  \ldots, e_m\}$ and $\{e_1 \ox 1, \ldots, e_m \ox 1\}$, respectively, 
  so that the
  algebras with involution $(B, \tau)$ and $(A \ox_ F N, \s \ox \id_N)$ are
  isomorphic via
  \[\xi : B \rightarrow A \ox_F N, \ e_i \mapsto e_i \ox 1.\]

  The set $\xi(\CS)$ is a prepositive cone on $(A\ox_F N,\s\ox\id_N)$ 
  over $Q$, so is
  included in a positive cone $\CQ$ on $(A \ox_ F N, \s \ox \id_N)$
  over $Q$. We have $\CP
  \subseteq \CS$, so that $\xi(\CP) \subseteq \CQ$, i.e., $\CP \ox 1 \subseteq
  \CQ$.
\end{proof}

\begin{prop}[{\cite[Proposition~6.7]{A-U-pos}}]\label{max=deg}\mbox{}
  Let $\CP$ be a positive cone on $(A,\s)$ over $P \in X_F$. There is
  $\ve \in \{-1,1\}$ such that for every $a \in \CP\cap A^\x$,
  $\sign^\mu_P \qf{a}_\s = \ve n_P(A,\s)$.
  In particular, $m_P(A,\s) = n_P(A,\s)$ for every $P \in X_F \setminus
  \Nil[A,\s]$. 
\end{prop}
\begin{proof}
    By Theorem~\ref{positive=max} there is $\ve \in \{-1,1\}$ such that
    $\sign^\mu_P \qf{a}_\s = \ve m_P(A,\s)$ for every $a \in \CP \cap A^\x$.
    Let $a \in \CP \cap A^\x$, cf \cite[Lemma~3.6]{A-U-pos}. 
    Since $m_P(A,\s)\leq n_P(A,\s)$ (cf. \cite[Proposition~4.4(iii)]{A-U-PS})
    and $\sign_P^\mu \qf{-a}_\s = - \sign_P^\mu \qf{a}_\s$,
    we prove the result by showing that $\sign^\mu_P
    \qf{a}_\s = \pm n_P(A,\s)$.

    Fix a basis $\cB = \{e_i\}_{i \in I}$ of $A$ over $F$, and let $F_0$ be the
    field obtained by adding the following elements to $\Q$, all determined
    with
    respect to the basis $\cB$: the structure constants of $A$,  
    the elements of the matrix of $\s$, 
    and the coordinates of $a$. Let $A_0$ be the $F_0$-algebra
    determined by these structure constants for a given basis $\cB_0 =
    \{e'_i\}_{i \in I}$ (i.e., we build the free $F_0$-algebra generated
    by the elements $e'_i$ and quotient out by the relations determined by
    the structure constants), 
    let $\s_0$ be the $F_0$-linear map on $A_0$ with the
    same matrix as $\s$, and let $a_0 \in A_0$ be the element with the same
    coordinates as $a$ (all with respect to $\cB_0$).  
    Let $\xi : A_0 \ox_{F_0} F \rightarrow A$, $e'_i \ox 1
    \mapsto e_i$. Since $A_0 \ox_{F_0} F$ and $A$ are $F$-algebras with the same
    structure constants with respect to $\{e'_i \ox 1\}_{i \in I}$ and
    $\{e_i\}_{i \in I}$, respectively, and the linear maps $\s_0 \ox 1$ and $\s$
    have the same matrices with respect to these same bases, the map $\xi$ is an
    isomorphism of $F$-algebras with involution from $(A_0 \ox_{F_0} F, \s_0 \ox
    \id)$ to $(A, \s)$, and $\xi(a_0 \ox 1) = a$. Therefore, we can assume for
    simplicity that $\xi$ is the identity map, so that
    $(A_0 \ox_{F_0} F, \s_0 \ox \id) = (A,\s)$, $a_0 \ox 1 = a$, and
    $A_0\subseteq A$.

    Let $P_0:= F_0\cap P$ and $\CP_0:=A_0 \cap \CP$. Note that $a_0 \in \CP_0$.
    By construction, $F_0$ is finitely generated over $\Q$, and $(F,P)$ is an
    ordered extension of $(F_0,P_0)$. By Theorem~\ref{embord} we have, 
    for any reference form
    $\mu_0$ for $(A_0, \s_0)$:
    \[\sign^{\mu_0}_{P_0} \qf{a_0}_{\s_0} = \sign^{\mu_0 \ox F}_{P} \qf{a_0
    \ox 1}_{\s_0 \ox 1} = \sign^{\mu_0 \ox F}_P \qf{a}_\s = \pm \sign^\mu_P
    \qf{a}_\s,\]
    where the final equality holds since a change of reference forms at most
    changes the sign of the signature, cf. 
     \cite[Proposition~3.3(iii)]{A-U-prime}. 
    In particular it suffices to prove
    the result for $(A_0, \s_0)$ and $P_0$, $\CP_0$, $a_0$. Therefore, we simply
    use the original notation and assume that $F$ is finitely generated over
    $\Q$.

  It follows by Proposition~\ref{arch-dense} that $P$ is in the closure of the 
  set
  of archi\-me\-dean orderings on $F$. Let $N$, $Q$ and $\CQ$ be as in
  Lemma~\ref{elementary}. By Proposition~\ref{dense-rc}, and using that
  invertible elements in a given positive cone have  signature 
  equal to $\pm m_P(A,\s)$
  by
  Theorem~\ref{positive=max}, we know that  for every $b \in \CQ\cap 
  (A\ox_F N)^\x$,
  \[
    \sign^{\mu \ox N}_Q \qf{b}_{\s \ox
  \id} = \pm n_Q(A\ox_F N, \s\ox\id)= \pm n_P(A,\s),
  \] 
  where the final equality follows from the fact that $(N,Q)$ is an ordered
  extension of $(F,P)$ and Lemma~\ref{npas} below.
  Since 
  $a \ox 1 \in (\CP \ox 1)\cap (A\ox_F N)^\x \subseteq \CQ \cap (A\ox_F N)^\x$, 
  it follows that
  \[\sign^\mu_P \qf{a}_\s = \sign^{\mu \ox N}_Q \qf{a \ox 1}_{\s \ox \id} =
  \pm n_P(A,\s).\qedhere\]
\end{proof}

\begin{lemma}\label{npas}
   Let $P\in X_F$ and let $(L,Q)$ be an ordered field extension of $(F,P)$.
   Then $n_P(A,\s)= n_Q (A\ox_F L, \s\ox \id)$.
\end{lemma}

\begin{proof}
  Recall that $n_P(A,\s)$ is defined via the isomorphism $A\ox_F F_P\cong 
  M_{n_P}(D_{F_P})$, with notation as in Sections~\ref{secqq} and 
  \ref{csa-sign}. Let $(L_Q,\tQ)$ be a real closure of $(L,Q)$.
  Observe that we may assume that $F_P\subseteq L_Q$, and thus that
  $D_{F_P}\ox_{F_P} L_Q\cong D_{L_Q}$, because $D_{F_P} \in \{F_P, 
  F_P(\sqrt{-1}), (-1,-1)_{F_P}\}$. Therefore,
  \[
    A\ox_F L_Q \cong A\ox_F F_P \ox_{F_P} L_Q \cong 
    M_{n_P}(D_{F_P})\ox_{F_P} L_Q \cong M_{n_P}(D_{L_Q}).
  \]
  Hence,
   $n_{\tQ}(A\ox_F L_Q, \s\ox\id) =n_P= n_P(A,\s)$.
  The result follows since $n_Q(A\ox_F L, \s\ox\id)=  n_{\tQ}(A\ox_F L_Q, 
  \s\ox\id)$ by definition.
\end{proof}

\begin{thm}[{\cite[Proposition~5.8]{A-U-pos}}]
  Let $P\in X_F$, let $(L,Q)$ be an ordered field extension of $(F,P)$ and let 
  $\CP$ be a 
  prepositive cone on $(A,\s)$ over $P$. Then $\CP\ox 1_L:=\{a\ox 1_L \mid 
  a\in\CP\}$ is contained in a prepositive cone on $(A\ox_F L, \s\ox\id)$
  over $Q$.
\end{thm}

\begin{proof}
Up to replacing $\CP$ by a positive cone that contains it, we may assume
that $\CP$ is a positive cone.
By Theorem~\ref{positive=max} we have $P \in X_F \setminus \Nil[A,\s]$, and 
so 
there is $a\in \Sym(A,\s)$ such that $\sign^\mu_P \qf{a}_\s \not = 0$,
cf. Remark~\ref{blobby}. 
Therefore,  by Theorem~\ref{embord},
\[\sign^{\mu \ox L}_Q \qf{a \ox 1}_{\s \ox \id} 
= \sign^\mu_P \qf{a}_\s \not = 0,\] 
and $Q \not \in \Nil[A \ox L, \s \ox \id]$. In particular there are
positive cones on $(A \ox L, \s \ox \id)$ and they are described by 
Theorem~\ref{positive=max}.
By Corollary~\ref{PUD},
\[\CP = \bigcup \{D_{(A,\s)} \qf{a_1, \ldots, a_k}_\s \mid k \in \N, \ a_1, 
\ldots, a_k \in \CM^\mu_P(A,\s)\}.\]
Therefore (using (a) and (b) below),
\begin{align*}
  \CP \ox 1_L &
  \subseteq \bigcup \{D_{(A \ox L,\s \ox \id)} \qf{a_1 \ox 1, 
  \ldots, a_k \ox 1}_\s \mid k \in \N,\\  
  &\strut\hspace{.55\linewidth} a_1, \ldots, a_k \in
                  \CM^\mu_P(A,\s)\} \\
      & \subseteq \bigcup \{D_{(A \ox L,\s \ox \id)} \qf{b_1, \ldots, b_k}_\s 
      \mid k \in \N, \\ 
  &\strut\hspace{.4\linewidth}    b_1, \ldots, b_k \in
                  \CM^{\mu \ox L}_Q(A \ox L, \s \ox \id)\} \\
      & = \CC_Q(\CM^{\mu \ox L}_Q(A \ox L, \s \ox \id)),
\end{align*}
which is a positive cone over $Q$ by Theorem~\ref{positive=max}, and where:
\begin{enumerate}[label=(\alph*)]
  \item The second inclusion uses the fact that $a \in \CM^\mu_P(A,\s)$ implies 
  $a \ox 1
  \in \CM^{\mu \ox L}_Q(A \ox L, \s \ox \id)$ which follows from the fact that
  $m_P(A,\s) = n_P(A,\s) = n_P(A \ox_F L, \s \ox \id) = m_P(A \ox_F L, \s \ox
  \id)$ by Proposition~\ref{max=deg} and Lemma~\ref{npas}.
  
  \item  The final equality follows from Corollary~\ref{PUD}.\qedhere
\end{enumerate}
\end{proof}

\end{document}